\documentclass{amsart}
\usepackage{amsmath, amsthm, amssymb}
\usepackage{verbatim}

\renewcommand{\d}{\partial}
\newcommand{\dbar}{\overline{\partial}}
\newcommand{\ddbar}{\sqrt{-1}\d\overline{\d}}

\newtheorem{thm}{Theorem}
\newtheorem{prop}[thm]{Proposition}
\newtheorem{lem}[thm]{Lemma}

\theoremstyle{definition}
\newtheorem{defn}[thm]{Definition}
\newtheorem{remark}[thm]{Remark}

\renewcommand{\[}{\begin{equation}}
\renewcommand{\]}{\end{equation}}

\title{K\"ahler-Einstein metrics along the smooth continuity method}
\author[V. Datar]{Ved Datar}
\author[G. Sz\'ekelyhidi]{G\'abor Sz\'ekelyhidi}
\address{Department of Mathematics, University of Notre Dame, 255
  Hurley, Notre Dame, IN 46556}

\begin{document}

\begin{abstract} We show that if a Fano manifold $M$ is K-stable with
  respect to special degenerations equivariant under a compact group
  of automorphisms, then $M$ admits a K\"ahler-Einstein metric. This
  is a strengthening of the solution of the Yau-Tian-Donaldson
  conjecture for Fano manifolds by Chen-Donaldson-Sun~\cite{CDS12}, and can be used to obtain new examples of K\"ahler-Einstein manifolds. 
  We also give analogous results for twisted K\"ahler-Einstein
  metrics and Kahler-Ricci solitons. 
\end{abstract}

\maketitle

\section{Introduction}
Let $M$ be a Fano manifold of dimension $n$. A basic problem in
K\"ahler geometry is whether $M$ admits a K\"ahler-Einstein
metric. The Yau-Tian-Donaldson conjecture~\cite{Yau93, Tian97, Don02},
confirmed recently by Chen-Donaldson-Sun~\cite{CDS12, CDS13_1,
  CDS13_2, CDS13_3}, says that $M$ admits a K\"ahler-Einstein metric if
and only if it is K-stable. In general it
seems to be intractable at present to check K-stability since in principle
one must study an infinite number of possible degenerations of $M$ to $\mathbf{Q}$-Fano
varieites. One goal of this paper is to study some situations with large symmetry
groups, where the problem reduces to checking a finite number of
possibilities. This can then be used to yield new examples of K\"ahler-Einstein
manifolds. 

Suppose then that a compact group $G$
acts on $M$ by holomorphic automorphisms. Our main theorem is the
following equivariant version of the result of Chen-Donaldson-Sun. 

\begin{thm}\label{thm:main}
  Suppose that $(M, K_M^{-1})$ is K-stable, with respect to
  special degenerations that are $G$-equivariant. Then $M$ admits a
  K\"ahler-Einstein metric. 
\end{thm}

Here a $G$-equivariant special degeneration is a special degeneration
$X\to\mathbf{C}$ in the sense of Tian~\cite{Tian97},
together with a holomorphic $G$ action which commutes
with the $\mathbf{C}^*$-action, preserves the fibers, and restricts to
the given action of $G$ on the generic fibers $X_t = M$ for $t\ne
0$. We also obtain an analogous result for K\"ahler-Ricci solitons,
and their twisted versions; see Definition~\ref{defn:Kstable} for
detailed definitions, and Proposition~\ref{prop:main} for the
most general result.  

An important special case is when $G$ is a torus. In particular if $M$
is a toric manifold, and $G= T^n$ is the $n$-torus, then
Proposition~\ref{prop:main} implies that we only need to check
special degenerations of the form $X = M\times \mathbf{C}$ to ensure the
existence of a K\"ahler-Einstein metric or K\"ahler-Ricci soliton on
$M$. In particular this recovers the result of Wang-Zhu~\cite{WZ04}
showing that all toric Fano manifolds admit a K\"ahler-Ricci
soliton. In addition we can recover the result of Li~\cite{Li11} on
the greatest lower bound on the Ricci curvature of toric Fano
manifolds.

A more interesting situation is when $G=T^{n-1}$, i.e. $M$ is
a complexity-one $T$-variety. In this case it is possible, in concrete
examples, to check all $G$-equivariant special degenerations of $M$,
and as a consequence we can obtain new examples of threefolds with
K\"ahler-Einstein metrics and K\"ahler-Ricci solitons. Work in
progress by Ilten-S\"uss~\cite{IS15}  suggests that we obtain
five new K\"ahler-Einstein threefolds. To our knowledge these are the first examples
where K-stability is used to obtain new K\"ahler-Einstein manifolds.

Our method of proof of Theorem~\ref{thm:main} is to use the classical
continuity path 
\[ \mathrm{Ric}(\omega_t) = t \omega_t + (1-t)\alpha \]
for $t\in [0,1]$ proposed by Aubin~\cite{Aub84}, and its analog for
K\"ahler-Ricci solitons studied by Tian-Zhu~\cite{TZ00}, and to show that if we
cannot find a solution for $t=1$, then there must be a $G$-equivariant
destabilizing special degeneration. In particular we obtain a new
proof of the result of Chen-Donaldson-Sun~\cite{CDS13_3}, without
using metrics with conical singularities. At the same time our
arguments are analogous to those in \cite{CDS13_3}, using also the adaptation of some
of those ideas to the smooth continuity method in \cite{Sz13_1}. 

A key advantage of the
smooth continuity path is that it allows one to work in a
$G$-equivariant setting. In contrast, in \cite{CDS13_3} one considers K\"ahler-Einstein
metrics singular along a smooth divisor $D\subset M$, and such a
divisor can not be $G$-invariant unless $G$ is finite (see
Song-Wang~\cite[Theorem 2.1]{SW12}). The
disadvantage of the smooth continuity path is that in effect one must
consider pairs $(V,\chi)$ of a variety $V$ together with a possibly
singular current $\chi$, as opposed to pairs $(V, D)$ of a variety and
a divisor. In \cite{CDS13_3} a destabilizing special degeneration
is obtained by applying the Luna slice
theorem, and for this we must restrict ourselves to a suitable finite
dimensional variety rather than the infinite dimensional space of
currents. For this the basic idea is to approximate a current $\chi$
by a sum of currents of integration along divisors. 

A brief outline of the paper is as follows. In
Section~\ref{sec:twistedKE} we collect some basic definitions and
results on twisted K\"ahler-Ricci solitons. The proof of the main result,
Proposition~\ref{prop:main}, will then be given in
Section~\ref{sec:mainthm}. We give some examples of the applications of our results
to toric manifolds and other manifolds of large symmetry group in
Section~\ref{sec:examples}.
In Section~\ref{sec:partialC0} we discuss
how to adapt the methods of \cite{Sz13_1} and \cite{PSS12} to obtain
the partial $C^0$-estimates along the continuity method for solitons. A
crucial point is the reductivity of the automorphism group of the
limiting variety. This essentially follows from the work of
Berndtsson~\cite{Ber13} as used in \cite{CDS13_3}, but
since we did not find the exact statement that
we need in the literature, we give a brief exposition in
Section~\ref{sec:reductive}.  

\section{Twisted K\"ahler-Ricci solitons}\label{sec:twistedKE}
Suppose that $W$ is a $\mathbf{Q}$-Fano manifold, with
log terminal singularities. In particular a power $K_{W_0}^r$ of the canonical
bundle on the regular set $W_0$ extends as a 
line bundle on $W$.  We say that a metric $h$ on $K_{W_0}^{-1}$ is
continuous on $W$, if the induced metric on $K_{W_0}^{-r}$ extends to
a continuous metric on $K_W^{-r}$. Fixing an open cover $\{U_i\}$ and local
trivializing holomorphic sections $\sigma_i$ of $K_V^{-r}|_{U_i\cap
  W_0}$, we will write 
\[ |\sigma_i|^2_{h^r} = e^{-r\phi_i}, \]
for continuous functions $\phi_i$ on $U_i$. We will write the metric
$h$ simply as $e^{-\phi}$ following the notation in
Berndtsson~\cite{Ber13}. In particular $e^{-\phi}$ defines a volume
form on $W_0$, given in a local chart $U_i$ by
\[ e^{-\phi} = |\sigma_i|^{2/r}_{h^r}(\sigma_i \wedge
\overline{\sigma}_i)^{-1/r}. \]
The log terminal condition says that this volume form has finite volume. We
write $\omega_\phi$ for the curvature current of the metric $e^{-\phi}$ on $W_0$, so
in our local charts $\omega_\phi = \ddbar\phi_i$. Since the potentials
$\phi_i$ are locally bounded, by Bedford-Taylor~\cite{BT76}
we can form the wedge product
$\omega_\phi^n$, which defines a measure on $W_0$, and also on $W$
extending it trivially. The metric $h_\phi$ is a weak K\"ahler-Einstein
metric if $\omega_\phi$ is a K\"ahler current, and we have
\[ e^{-\phi} = \omega_\phi^n. \]

Berman and Witt-Nystr\"om~\cite{BW14} have studied the analogous notion
of weak K\"ahler-Ricci solitons. Suppose that $v$ is a holomorphic
vector field on $W_0$, whose imaginary part generates the action of a
torus $T$ on $W$ (see
Berman-Boucksom-Eyssidieux-Guedj-Zeriahi~\cite[Lemma 5.2]{BBEGZ11} to
see that one obtains an action on $W$). 
A K\"ahler-Ricci soliton on $(W,v)$ is a $T$ invariant
continuous metric $e^{-\phi}$, smooth on $W_0$ with positive curvature current
$\omega_\phi$ satisfying
\[ e^{-\phi} = e^{\theta_v}\omega_\phi^n. \]
Here $e^{\theta_v}\omega_\phi^n$ is a measure defined in \cite{BW14}
for general $\phi$. If $\phi$ is smooth, then 
$\theta_v$ is simply a Hamiltonian function for the vector field $v$,
satisfying
\[ L_v \omega_\phi = \ddbar \theta_v, \]
with the normalization
\[ \label{eq:aa1} \int_{W_0} e^{\theta_v} \omega_\phi^n = \int_{W_0}
\omega_\phi^n = V. \]
In particular $\theta_v$ depends on $\phi$. 
For continuous metrics $h_\phi$ (or more general metrics with positive
curvature current), the measure constructed in \cite{BW14} still
satisfies the normalization \eqref{eq:aa1}. In addition by
\cite[Corollary 2.9]{BW14} we have some fixed constant $C$ (depending only
on $M, v$), such that 
\[ \label{eq:aaa3} C^{-1} \omega_\phi^n \leq e^{\theta_v} \omega_\phi^n \leq  C
\omega_\phi^n. \]

We now use this to 
define the twisted analogs of K\"ahler-Ricci solitons, which arise
naturally along the continuity method. 
Suppose that $e^{-\psi}$ is another metric on $K_{W_0}^{-1}$ which in
our local charts is given by plurisubharmonic functions $\psi_i \in
L^1_{loc}(U_i\cap W_0)$. 

\begin{defn} For $t\in (0,1)$ we say that the pair $(W,
(1-t)\psi)$ is $klt$, if in each chart $U_i\cap W_0$ the function
$e^{-\psi_i}$ is integrable, with respect to the volume form
$(\sigma_i \wedge \overline{\sigma}_i)^{-1/r}$. We will on occasion
write $(W, (1-t)\omega_\psi)$ for the pair, where as before
$\omega_\psi$ is the curvature of $e^{-\psi}$. 
\end{defn}

Equivalently we can
think of $e^{-t\phi - (1-t)\psi}$ as a volume form on $W_0$ with
$e^{-\phi}$ being a continuous metric as above. The $klt$ condition is then
\[ \int_{W_0} e^{-t\phi - (1-t)\psi} < \infty. \]

\begin{defn} \label{defn:tKRS} A twisted K\"ahler-Ricci soliton on the triple $(W,
  (1-t)\psi, v)$, where $v$ is a holomorphic vector field as above,
  is a continuous metric $e^{-\phi}$ such that
  \[ \label{eq:soleq} e^{-t\phi - (1-t)\psi} = e^{\theta_v} \omega_\phi^n. \]
This equation is interpreted as an equality of measures on $W_0$, and
in particular $e^{-\phi}$ here need not be smooth on $W_0$, so
$e^{\theta_v}\omega_\phi^n$ is the measure defined by
Berman-Witt-Nystr\"om~\cite{BW14}. Note
that the existence of such a metric implies that $(W, (1-t)\psi)$ is
$klt$. When $t=1$ or $v=0$, we will simply omit the corresponding
term in the triple. So we can talk about a K\"ahler-Einstein metric on
$W$, a twisted K\"ahler-Einstein metric on
$(W, (1-t)\psi)$, or a K\"ahler-Ricci soliton on $(W,v)$. 
\end{defn}

\begin{remark}\label{rem:smoothsoliton}
If $W, \phi, \psi$ are smooth, then the twisted K\"ahler-Ricci soliton
equation is equivalent (up to adding a constant to $\phi$) to 
\[ \label{eq:aa2} \mathrm{Ric}(\omega_\phi) - L_v \omega_\phi = t\omega_\phi +
(1-t)\omega_\psi, \]
which is the natural continuity path for finding K\"ahler-Ricci
solitons, used by Tian-Zhu~\cite{TZ00} for instance. 

Even when $W$ is
normal and $\phi$ is only continuous, it is useful to have an equation
for twisted K\"ahler-Ricci solitons in the form \eqref{eq:aa2}. For
this the extra condition needed is that the measure
$e^{\theta_v}\omega_\phi^n$ defines a singular metric $e^{-\tau}$ on
$K_W$, with $\tau\in L^1_{loc}$. Then $\phi$ defines a twisted
K\"ahler-Ricci soliton on $(W, (1-t)\psi, v)$ if 
\[ \label{eq:aa3} \omega_\tau = t\omega_\phi + (1-t)\omega_\psi, \]
where $\omega_\tau$ is the curvature of $e^{-\tau}$. 
Note that by an argument similar
to Berman-Boucksom-Eyssidieux-Guedj-Zeriahi~\cite[Proposition
3.8]{BBEGZ11}, 
if $e^{-\tau}$ is only defined
outside a subset $S\subset W$ with $(2n-2)$-dimensional Hausdorff
measure $\Lambda_{2n-2}(S) = 0$, and Equation~\eqref{eq:aa2} holds on
$W\setminus S$, then $e^{-\phi}$ is a twisted K\"ahler-Ricci
soliton. Indeed in this case $e^{-\tau}$ extends as a singular metric with positive
curvature current over all of $W$ (see Harvey-Polking~\cite[Theorem
1.2]{HP75}, Demailly~\cite{Dem85}),
and then \eqref{eq:aa3} implies that up to modifying $\psi$ by a
constant, we must have
\[ e^{\theta_v}\omega_\phi^n = e^{-\tau} = e^{-t\phi - (1-t)\psi}, \]
since \eqref{eq:aa3} implies that $f = \tau - t\phi - (1-t)\psi$ is a
global $L^1$ function with $\ddbar f =0$ on $W$.  
\end{remark}

We need the following result, generalizing the
classical results of Bando-Mabuchi~\cite{BM85} and
Matsushima~\cite{Mat57},  which are essentially
contained in Berndtsson~\cite{Ber13},
Boucksom-Berman-Eyssidieux-Guedj-Zeriahi~\cite{BBEGZ11}, 
Berman-Witt-Nystr\"om~\cite{BW14} and
Chen-Donaldson-Sun~\cite{CDS13_3}.  We will give an outline proof in
Section~\ref{sec:reductive}. 
\begin{prop} \label{prop:uniqueness}
  Suppose that $e^{-\phi_0}, e^{-\phi_1}$ are two twisted
  K\"ahler-Ricci solitons on $(W, (1-t)\psi,v)$. Then there exists a
  holomorphic vector field $w$ on $W$, commuting with $v$ and
  satisfying $\iota_w \omega_\psi
  = 0$, such that the biholomorphisms $F_t:W\to W$ induced by $w$
  satisfy $F_1^*(\omega_{\phi_1}) = \omega_{\phi_0}$.
  In addition $L_{\mathrm{Im}\,w}\omega_\phi = 0$. 
\end{prop}

\begin{defn}\label{defn:aut}
  For any triple $(W, (1-t)\psi, v)$ we define the Lie algebra
\[ \mathfrak{g}_{W, \psi, v} = \{ w\in H^0(TW)\,:\,
\iota_w\omega_\psi = 0 \text{ and } [v,w]=0\}. \]

  
  \noindent As before, we may omit $\psi$ or $v$ from the notation if $t=1$
  or $v=0$. In particular $\mathfrak{g}_W = H^0(TW)$. We will also
  write $\mathfrak{g}_{W, \beta} = \mathfrak{g}_{W, \psi}$ if $\beta =
  \omega_\psi$ is the curvature of $e^{-\psi}$. Using a
  projective embedding into $\mathbf{P}^N$, we can realize $\mathfrak{g}_{W, \psi, v}$ as a
  subalgebra of $\mathfrak{sl}(N+1,\mathbf{C})$. 
\end{defn}

  Note that for example $\mathfrak{g}_{W, \psi}$ is trivial if $\omega_\psi$ is
  strictly positive and $t < 1$. In fact Berndtsson~\cite[Proposition 8.2]{Ber13}
  implies that if $e^{-\psi}$ is integrable, then
  $\mathfrak{g}_{W, \psi}$ is trivial. In our application, when $(W,
  (1-t)\psi)$ is klt,  $e^{-(1-t)\psi}$ will be integrable, but
  $e^{-\psi}$ will typically not be.

  Note also that the Lie group with Lie algebra $\mathfrak{g}_{W,
    \psi}$ will usually be strictly smaller than the identity component of the group of
  biholomorphisms of $W$ preserving $\omega_\psi$. The difference comes
  from the fact that if $v$ is a real vector field then
  $L_v\omega_\psi$ does not imply $L_{Jv}\omega_\psi$ for the
  complex structure $J$, whereas our Lie algebra above is
  automatically closed under multiplication by $\sqrt{-1}$. On the
  other hand when $\omega_\psi = [D]$ is the current of integration along a
  divisor, then $\mathfrak{g}_{W, \psi}$ coincides with the vector
  fields on $W$ parallel to $D$. Indeed $\iota_v [D]=0$ is equivalent to $v$ being parallel to
  $D$ along the smooth part of $D$.

The following theorem generalizes \cite[Theorem 6]{CDS13_3}, which in
turn is a generalization of Matsushima's theorem~\cite{Mat57} on the
reductivity of the automorphism group of a K\"ahler-Einstein
manifold. We will give the proof in Section~\ref{sec:reductive}. 

\begin{prop} \label{prop:reductive}
  Suppose that $(W,(1-t)\psi,v)$ admits a twisted K\"ahler-Einstein
  metric $e^{-\phi}$. Then $\mathfrak{g}_{W, \psi, v}$ is reductive. In
  addition if $G$ is a group of biholomorphisms of $W$, fixing
  $\omega_\phi$ and $v$, then the centralizer $(\mathfrak{g}_{W, \psi,
    v})^G$ 
  is also reductive. 
\end{prop}

We finally recall some properties of the ``twisted'' Futaki invariant,
generalizing the log-Futaki invariant in \cite{CDS13_3} and the modified
Futaki invariant of Tian-Zhu~\cite{TZ00}. For a smooth metric $e^{-\phi}$
on $K_{W_0}^{-1}$ we define
\[\label{eq:Futdef} 
\begin{aligned} \mathrm{Fut}_{(1-t)\psi, v}(W, w) = \mathrm{Fut}_v(W,w) -
\frac{1-t}{V}&\left[ \int_W \theta_w(e^{\theta_v} -
  1)\,\omega_\phi^n\right. \\
 &\quad \left.+ n
  \int_W \theta_w(\omega_\psi - \omega_\phi)\wedge
  \omega_\phi^{n-1}\right], 
\end{aligned}\]
where $\mathrm{Fut}_v(W,w)$ is Tian-Zhu's modified Futaki invariant,
which we write in the form
\[ \label{eq:modFut} \mathrm{Fut}_v(W,w) = \frac{1}{V}\int_W \theta_w
e^{\theta_v}\,\omega_\phi^n - \frac{\int_W \theta_w e^{-\phi}}{\int_W
  e^{-\phi}}.\]
This is shown to be equivalent to Tian-Zhu's definition by
He~\cite{He12}. One can check by direct calculation that our definition of the
twisted Futaki invariant is independent of the metric $e^{-\phi}$,
remembering that $\iota_w \omega_\psi =0$. 

We will on occasion write $\mathrm{Fut}_{(1-t)\omega_\psi, v}$ instead
of $\mathrm{Fut}_{(1-t)\psi, v}$, when the curvature of $e^{-\psi}$ is
more natural. We will need the following:
\begin{prop}\label{prop:Futvanish}
  If $(W,(1-t)\psi,v)$ admits a twisted K\"ahler-Ricci soliton, then
  \[\mathrm{Fut}_{(1-t)\psi, v}(W,w) = 0\]
  for all $w \in\mathfrak{g}_{W, \psi, v}$. 
\end{prop}

If the twisted K\"ahler-Ricci soliton had smooth potential $\phi$, at
least on $W_0$, then this would follow directly from the
definitions. In general we obtain the result by relating 
the twisted Futaki invariant to the twisted Ding functional, and using
that the twisted Ding functional is bounded
below if there exists a twisted KR-soliton. This is analogous to an
argument in \cite{CDS13_3}, and the proof will be given in
Section~\ref{sec:reductive}.

\subsection*{Twisted stability}

Suppose now that $M$ is a smooth Fano manifold,
with a holomorphic vector field $v$ such that $\mathrm{Im}\,v$
generates a torus $T$. Suppose that $G$ is a compact group of
automorphisms of $M$, containing $T$. We embed $M\subset \mathbf{P}^N$ using
$G$-invariant sections of $K_M^{-m}$ for some $m$. Let $\alpha =
\frac{1}{m} \omega_{FS}|_M$, which we can write as the curvature of a
smooth metric $e^{-\psi_\alpha}$ on $K_M^{-1}$ in the
notation above. This metric will then be $G$-invariant.
It was shown by Dervan~\cite{Der14} that
twisted K-stability is a necessary condition for the existence of a
twisted KE metric on $(M, (1-t)\alpha)$, while a corresponding stability
notion for K\"ahler-Ricci solitons was developed by
Berman-Witt-Nystr\"om~\cite{BW14}. We can combine these ideas to
obtain a stability notion for twisted K\"ahler-Ricci solitons as
follows. 

The vector field $v$ on $M$ is the restriction of a holomorphic vector field on
$\mathbf{P}^N$, which we will also denote by $v$. The imaginary part
$\mathrm{Im}\,v$ corresponds to a matrix in $\mathfrak{u}(N+1)$, with
eigenvalues $\mu_i$, so that $v$ has Hamiltonian function
\[ \theta_v = \frac{\sum_i \mu_i |Z_i|^2}{\sum_i |Z_i|^2}\]
for suitable homogeneous coordinates $Z_i$. We assume that $\theta_v$
is normalized as before (i.e. $e^{\theta_v}$ has average 1 on $M$).

Under our embedding 
the group $G$ above can be thought of as a subgroup of $U(N+1)$. Suppose that we
have a $\mathbf{C}^*$-action $\lambda \subset GL(N+1,\mathbf{C})^G$,
generated by a vector field $w$ on $\mathbf{P}^N$, where $GL(N+1,
\mathbf{C})^G$ denotes the centralizer of $G$. 
Suppose that the central fiber $W = \lim_{t\to 0}\lambda(t)\cdot M$ is a
$\mathbf{Q}$-Fano variety. We can also take the limit 
\[ \beta = \lim_{t\to 0} \lambda(t)\cdot \alpha, \]
which is a closed positive current on $W$. 
The $\mathbf{C}^*$-action $\lambda$ defines a
special degeneration (in the terminology of
Tian~\cite{Tian97}), and its twisted Futaki invariant is defined to be
\[ \label{eq:tFdef}\mathrm{Fut}_{(1-t)\alpha, v}(M, w) = \mathrm{Fut}_{(1-t)\beta,
v}(W,w), \]
again omitting $\alpha, \beta$ or $v$ if $t=1$ or $v=0$.

\begin{defn}\label{defn:Kstable}
  We say that the triple $(M, (1-t)\alpha, v)$ is K-semistable
  (with respect to $G$-equivariant special degenerations),
  if $\mathrm{Fut}_{(1-t)\alpha, v}(M, w) \geq 0$ for all $w$ as
  above. The triple is K-stable if in addition
  equality holds only when $(W, (1-t)\beta)$ is biholomorphic to $(M,
  (1-t)\alpha)$, i.e. the pairs are in the same $GL^G$-orbit. 
\end{defn}

The terminology more consistent with existing literature would be
``twisted modified K-polystable'', but we hope no confusion is caused by
simply using the terminology ``K-stable''. 
Dervan~\cite{Der14} showed that if $(M, (1-t)\alpha)$ admits a twisted
K\"ahler-Einstein metric then it is K-stable, while
Berman-Witt-Nystr\"om showed that if $(M, v)$ admits a K\"ahler-Ricci
soliton, then it is K-stable in the sense of the above definition. We
expect that one can combine the arguments to show that if the triple
$(M, (1-t)\alpha, v)$ admits a twisted K\"ahler-Ricci soliton, then it
is K-stable, but we will not pursue that here. Our main result is a
result in the converse direction, the proof of which will be given in
Section~\ref{sec:mainthm}. 

\begin{prop}\label{prop:main}
 If $(M, (1-s)\alpha, v)$ is K-semistable for all $G$-equivariant
 special degenerations, then $(M, (1-t)\alpha, v)$ admits a
 twisted K\"ahler-Ricci soliton for all $t < s$.
 In addition if $(M, v)$ is K-stable, then $(M,v)$
 admits a K\"ahler-Ricci soliton.
\end{prop}

Note that we also expect that if $(M, (1-t)\alpha, v)$ is K-stable, then $(M,
(1-t)\alpha, v)$ admits a twisted K\"ahler-Ricci soliton, however this
does not quite follow from our arguments. 

A key ingredient in our arguments is a comparison of the twisted
and untwisted Futaki invariants and from \eqref{eq:Futdef} it follows that 
\[ \label{eq:tdepend}\begin{aligned}
  &\mathrm{Fut}_{(1-t)\alpha, v}(M, w) = \mathrm{Fut}_{v}(M,w)
  \\
&\quad  -\frac{1-t}{V}\left[ \int_W \theta_w (e^{\theta_v}-1)\,\omega_\phi^n 
  +n\int_W \theta_w(\beta - \omega_\phi)\wedge \omega_\phi^{n-1}\right].
\end{aligned}\]
Recall here that $(W, \beta)$ is the limit of the pair $(M, \alpha)$
under the $\mathbf{C}^*$-action generated by $w$. 
The following result builds on work in \cite{LSz13} and
Dervan~\cite{Der14}.

\begin{prop}\label{prop:thetaform} Using the same setup as above,
  we have the formula
 \[  \label{eq:c1} \frac{1}{V} \int_W \theta_w \Big[ (n+1)\omega_\phi - n
  \beta\Big]\wedge \omega_\phi^{n-1} = \max_W\theta_w. \]
\end{prop}
We will give the proof below, after Lemma~\ref{lem:inttheta}. For now
note that as a consequence we have
\[ \mathrm{Fut}_{(1-t)\alpha, v}(M, w) = \mathrm{Fut}_v(M,w)
  + \frac{1-t}{V}\int_W (\max_W \theta_w - \theta_w)
  e^{\theta_v}\,\omega_\phi^n. \]
In particular the difference is always positive, and is equal to zero
only if $\theta_w$ is constant on $W$, i.e. if we had a trivial
degeneration. Note also that the right hand side is independent of the
choice of metric $\alpha$ on $M$, however as discussed in
\cite{Sz12_1} (and can be seen from the proof below),
if one replaces $\alpha$ by the current of integration
along a divisor, leading to the notion of log K-stability used in
\cite{CDS13_3}, the twisted Futaki invariant might drop for special
divisors. 

For the proof of Proposition~\ref{prop:thetaform}, and also for later
use we will need to represent $\alpha$ as an
integral of currents of integration along divisors on $M$. The
formula \eqref{eq:c1} is invariant under scaling $\omega_\phi$ and
$\omega_\psi$, and so to simplify notation we will assume that the
cohomology classes $[\alpha], [\omega_\phi]$ coincide with the
classes of the hyperplane divisors $M\cap H, W\cap H$. In particular
we then have $V=1$.  We will also normalize the
Fubini-Study metric $\omega_{FS}$ on $\mathbf{P}^N$ to represent the
same cohomology class as $[H]$.

Let us write $\mathbf{P}^{N*}$ for the
dual projective space of hyperplanes. Since $\alpha$ is the
restriction of $\omega_{FS}$ to $M$, we have (see
e.g. Shiffman-Zelditch~\cite{SZ99}) 
\[ \alpha = \int_{\mathbf{P}^{N^*}} [M\cap H]\, d\mu(H), \]
where $d\mu$ is simply the Fubini-Study volume form, scaled to have
volume $1$. It follows that the limit $\beta = \lim_{t\to 0}
\lambda(t)_*\alpha$ is given by
\[ \beta = \int_{\mathbf{P}^{N^*}} [W \cap H_0]\, d\mu(H), \]
where for each hyperplane $H$ we wrote
\[ H_0 = \lim_{t\to 0} \lambda(t)\cdot H. \]
In this formula for the limit $\beta$ it is important that $W$ is not 
contained in a
hyperplane, otherwise we would not necessarily have the relation
\[  \lim_{t\to 0} \lambda(t)\cdot (M\cap H) = (\lim_{t\to 0}
\lambda(t)\cdot M) \cap (\lim_{t\to 0}\lambda(t)\cdot H),\]
used above.  It follows that 
\[ \label{eq:4} \int_{W} \theta_w \beta \wedge \omega_{FS}^{n-1} =
\int_{\mathbf{P}^{N^*}} \int_{W\cap H_0} \theta_w\,\omega_{FS}^{n-1}\,d\mu(H). \]
A key point is that there is a subspace $P_w\subset \mathbf{P}^{N^*}$,
depending on $w$, such that for all $H\not\in P_w$ the integral 
\[ \int_{W\cap H_0} \theta_w\,\omega_{FS}^{n-1} \]
has the same value. The following lemma gives a formula for this
integral, and in particular shows this independence. This formula is
essentially contained in \cite[proof of Theorem 12]{LSz13}, and was
made more explicit by Dervan~\cite{Der14}. 

\begin{lem}\label{lem:inttheta}
  Let us normalize the Fubini-Study metric  so
  that $[\omega_{FS}] = [H]$ in $H^2(\mathbf{P}^N)$. Then there is a
  subspace $P_w\subset \mathbf{P}^{N^*}$ such that for $H\not\in P_w$ we have
\[ \int_{W \cap H_0} \theta_w\,\omega_{FS}^{n-1} = \frac{1}{n}\int_W \Big[(n+1)
\theta_w - \max_W \theta_w\Big]\,\omega_{FS}^n. \]
\end{lem}
\begin{proof}
  Let us write $R = \bigoplus R_k$ for the graded coordinate ring of
  $W$. In suitable homogeneous coordinates the function $\theta_w$ on
  $\mathbf{P}^N$  is given by
  \[ \theta_w(Z) = \frac{ \sum_i \mu_i |Z_i|^2}{\sum_i |Z_i|^2}, \]
  where the $\mu_i$ are the weights of the $\mathbf{C}^*$-action
  $\lambda(t)$ induced by $\theta_w$, on the linear functions $R_1$. For a
  generic hyperplane $H$, the limit $H_0 = \lim_{t\to
    0}\lambda(t)\cdot H$ has equation $Z_{max} = 0$, where $\mu_{max}$
  is the largest weight (if there are several equal largest weights,
  then $Z_{max}$ can denote any of the corresponding
  coordinates). Indeed this is the case for all hyperplanes not
  passing through the set where $\theta_w$ achieves its maximum.
  This can be seen from the fact that the effect of
  acting by $\lambda(t)$ as $t\to 0$ is the same as flowing along the
  negative gradient flow of $\theta_w$.

  Denoting by $S =\bigoplus S_k$ the graded coordinate ring of $W\cap
  H_0$, we have $S = R / Z_{max}R$, i.e. $S_k = R_k /
  Z_{max}R_{k-1}$. Let us write $w_k$ for the total weight of the
  action $\lambda$ on $R_k$, and $w_k'$ for the weight of the action
  on $S_k$. From the equivariant Riemann-Roch theorem we have
  \[ \dim R_k = k^n \int_W \frac{\omega_{FS}^n}{n!} + O(k^{n-1}), \]
  and 
  \[ w_k = k^{n+1} \int_W \theta_w\,\frac{\omega_{FS}^n}{n!} + c k^n +
  O(k^{n-1}) \]
  for some constant $c$.  Similarly
  \[ \label{eq:3} w_k' = k^n \int_{W\cap H_0} \theta_w\, \frac{\omega_{FS}^{n-1}}{(n-1)!}
  + O(k^{n-1}). \]
  From the description $S_k = R_k / Z_{max}R_{k-1}$ we get
  \[ \begin{aligned}
    w_k' &= w_k - w_{k-1} - \mu_{max} \dim R_{k-1} \\
    &= (n+1)k^n \int_{W} \theta_w\,\frac{\omega_{FS}^n}{n!} - \mu_{max} k^n
    \int_W \frac{\omega_{FS}^n}{n!}.
  \end{aligned}\]
  Combining this with \eqref{eq:3} we get
  \[ \int_{W\cap H_0} u\,\omega_{FS}^{n-1} =
  \frac{1}{n} \int_W \Big[ (n+1)\theta_w - \mu_{\max}\Big]\,\omega_{FS}^n. \]
  The fact that $W$ is invariant under the action of $\lambda(t)$ and
  not contained in a hyperplane
  implies that $\max_W u = \mu_{max} = \max_{\mathbf{P}^N} u$. 
\end{proof}

Proposition~\ref{prop:thetaform} follows from this lemma together with
the formula \eqref{eq:4}. Indeed, the lemma together with \eqref{eq:4}
implies that
\[ \label{eq:c2} \int_W \theta_w\,\beta \wedge \omega_{FS}^{n-1} = \frac{1}{n}
\int_W \Big[ (n+1)\theta_w - \max_W \theta_w\Big]\,\omega_{FS}^n, \]
since the set of hyperplanes in $P_w$ has measure zero. At the same
time, in \eqref{eq:c1} we can replace $\omega_\phi$ with the
restriction of $\omega_{FS}$ to $W$. Note that this
will change the function $\theta_w$, but the difference of the two
sides of \eqref{eq:c1} remains the same. The formula \eqref{eq:c1} in
Proposition~\ref{prop:thetaform} then follows immediately from \eqref{eq:c2}.

\section{Proof of the main result} \label{sec:mainthm}
In this section we give the proof of our main result,
Proposition~\ref{prop:main}. The setup is that we have a smooth Fano
manifold $M$ with the holomorphic action of a compact group $G$. We
have a $G$-invariant K\"ahler metric $\alpha\in c_1(M)$, and for
simplicity we assume that $\alpha$ is the restriction of
$\frac{1}{m}\omega_{FS}$ to $M$, under an embedding $M\subset
\mathbf{P}^{N_m}$ using a basis of sections of $K_M^{-m}$, for some $m
> 0$. We are also given a vector field $v$ on $M$, invariant under the
action of $G$. In order to find a K\"ahler-Ricci soliton on $(M,v)$ we
try to solve the equations 
\[\label{eq:KRScont} \mathrm{Ric}(\omega_t) - L_v \omega_t = t\omega_t + (1-t)\alpha,\]
for $t\in [0,1]$. From Zhu~\cite{Zhu00} we know that there is a
solution for $t=0$ and by Tian-Zhu~\cite{TZ00} the possible values of
$t$ form an open set. We therefore have a solution for $t\in [0,T)$ and we need
to understand the limit of a sequence of solutions as $t\to T$. 

\subsection{The case $T < 1$. }
We first focus on the case $T < 1$, and we assume that the triple
$(M,(1-s)\psi_\alpha, v)$ is K-stable with respect to $G$-equivariant
special degenerations, for some $s\in (T,1]$,  We show
that in this case the continuity method cannot blow up at time
$T$, i.e. we can solve our equation for $t=T$ as well. 
The strategy is the same as that in \cite{CDS13_3}. 

We first show that along
a sequence $t_k \to T$, the Gromov-Hausdorff limit of $(M,
\omega_{t_k})$ has the structure of a $\mathbf{Q}$-Fano variety $W$,
together with a metric $\psi$ on $K_W$, and a vector field $v$ such
that the triple $(W, (1-T)\psi, v)$ admits a twisted K\"ahler-Ricci
soliton. We then need to show that $W$ is the central fiber of a
special degeneration for $M$.  
One difficulty, when comparing this to the analogous result in
\cite{CDS13_3}, is that we are not able to show that the pair $(W, (1-T)\psi)$
is the central fiber of a special degeneration for $(M,(1-T)\psi_\alpha)$ since
we are not able to use the Luna slice theorem on the infinite
dimensional space of pairs consisting of a variety and a positive
current. Instead we use an argument approximating $\alpha$ with a
convex combination of hyperplane sections. 

The key ingredient to understanding the Gromov-Hausdorff limit of a
sequence $(M, \omega_{t_k})$ is the partial $C^0$-estimate, first
introduced by Tian~\cite{Tian90}. This was
established in \cite{Sz13_1} in the case when $v=0$, using the method in
Chen-Donaldson-Sun~\cite{CDS13_2}, and it was shown by 
Phong-Song-Sturm~\cite{PSS12} for
K\"ahler-Ricci solitons (i.e. $v$ is non-zero, but $t=1$),
generalizing the work of Donaldson-Sun~\cite{DS12}. A modest
combination and generalization of
these ideas gives the analogous result for the equation
\eqref{eq:KRScont}, and we will give a brief outline of the necessary
changes in Section~\ref{sec:partialC0}. 

For each $t$, the metric $\omega_t$ introduces Hermitian inner products
on $H^0(K_M^{-m})$ for all $m > 0$, moreover these inner products are
$G$-invariant (by the uniqueness
of solutions to \eqref{eq:KRScont} for $t < 1$). 
The partial $C^0$-estimate says that we can find a
uniform $m$, and $\kappa > 0$, independent of $t$, such that an
orthonormal basis $\{s_0,\ldots, s_{N_{m}}\}$ of $H^0(K_M^{-m})$ satisfies
\[ \kappa < \sum_{i=0}^{N_{m}} |s_i|^2(x) < \kappa^{-1} \]
for all $x\in M$. Let us write $N=N_m$ for this choice of $m$ from now. 

Let us now write $V_t = H^0(K_M^{-m})$ for the unitary
$G$-representation, with metric induced by $\omega_t$. Note that $V_t$
are equivalent $G$-representations, and hence they are unitarily
equivalent as well. It follows that we have $G$-equivariant unitary
maps $f_t : V_0\to V_t$. In other words 
if we pick an orthonormal basis $\{s_0,\ldots, s_N\}$ for
$H^0(K_M^{-m})$ with respect to the metric $\omega_0$, then for all $t
> 0$ we can find an orthonormal basis $\{ s_0^{(t)},\ldots,
s_N^{(t)}\}$ with respect to $\omega_t$, by applying the map $f_t$.
Using these bases, we have embeddings $F_t : M \to
\mathbf{P}^{N}$, such that for $s\ne t$ we have $F_s = \rho\circ
F_t$ with $\rho\in GL(N+1)^G$, i.e. $\rho$ commutes with $G$. In
particular the vector field $(F_t)_*v$ along the image $F_t(M)$ is
induced by a fixed holomorphic vector field $v$ on $\mathbf{P}^N$,
since $v$ is $G$-invariant. 

We can choose a subsequence $t_k\to T$, such that $F_{t_k}(M)$ converges to
a limit $W\subset \mathbf{P}^N$, and as shown in
Donaldson-Sun~\cite{DS12}, the partial $C^0$-estimate implies, up to
replacing $m$ by a multiple, 
that $W$ is a normal $\mathbf{Q}$-Fano variety, homeomorphic to the Gromov-Hausdorff
limit $Z$ of the sequence $(M, \omega_{t_k})$. Moreover the maps
$F_{t_k}:M \to \mathbf{P}^N$ converge to a Lipschitz map
$F_T:Z\to \mathbf{P}^N$ under this Gromov-Hausdroff convergence,
such that $F_T:Z \to W$ is a homeomorphism.  Note that by choosing a
further subsequence we can assume that the currents
$(F_{t_k})_*\alpha$ converge weakly to a current $\beta$, which is
necessarily supported on $W$ and is invariant under the action of
$\mathrm{Im}\,v$. Let us write $\beta$ as the curvature $\omega_\psi$
of a singular metric $e^{-\psi}$ on $K_W^{-1}$. We can similarly
define a weak limit $\omega_T$ of the metrics
$(F_{t_k})_*(\omega_{t_k})$, which is also supported on $W$. Note that
if we write
\[ \omega_{t_k} = \frac{1}{m} (F_{t_k})^*\omega_{FS} + \ddbar
\phi_k, \]
then the partial $C^0$-estimate implies that we have  bounds
$|\phi_k|, |\nabla \phi_k|_{\omega_{t_k}} < C$. This in particular
implies that the $\phi_k$ converge to a Lipschitz function $\phi_T$ on
$(Z, d_Z)$, and since $Z$ is homeomorphic to $W$, this means that
$\phi_T$ is continuous on $W$ (using the topology induced from
$\mathbf{P}^N$). This implies that $\omega_T$ is the curvature of a
continuous metric $e^{-\phi_T}$ on $K_W^{-1}$ (recall that we might need
to take a power $K_W^{-m}$ here). 
We need the following.

\begin{prop}\label{prop:limiteq}
  The triple $(W, (1-T)\psi, v)$ admits a twisted
  K\"ahler-Ricci soliton, and in particular $(W, (1-T)\psi)$ has $klt$
  singularities. In fact the twisted K\"ahler-Ricci soliton is given
  by the metric $e^{-\phi_T}$. 
\end{prop}
\begin{proof}
  Let us decompose the Gromov-Hausdroff limit as $Z = \mathcal{R}\cup
  \mathcal{D} \cup \mathcal{S}_2$. Here $\mathcal{R}$ is the regular
  set, and $\mathcal{D}$ is the set of points which admit a tangent
  cone of the form $\mathbf{C}^{n-1}\times \mathbf{C}_\gamma$, where
  $\mathbf{C}_\gamma$ is the standard cone with cone angle
  $2\pi\gamma$. See Section~\ref{sec:partialC0} for more details. 
  From the results of Cheeger-Colding~\cite{CC97} and
  Cheeger-Colding-Tian~\cite{CCT02} we know that $S_2$ is a closed set of
  Hausdorff dimension at most $2n-4$. Since $F_T$ is Lipschitz, we
  know that $F_T(S_2)$ is also a closed set with Hausdorff dimension at most
  $2n-4$. Let us write $W' = W_0 \setminus F_T(S_2)$, where as before
  $W_0$ is the regular part of the algebraic variety $W$. We will
  construct the twisted K\"ahler-Ricci soliton on $W'$. As explained
  in Remark~\ref{rem:smoothsoliton}, it is enough to show that the measure
  $e^{\theta_v}\omega_T^n$ corresponding to the metric $e^{-\phi_T}$
  defines a singular metric $e^{-\tau}$ on $K_{W'}$ with $\tau\in
  L^1_{loc}$ such that its curvature satisfies
\[ \label{eq:aa4} \omega_\tau = T\omega_T + (1-T)\psi. \] 

  To simplify notation we will identify $Z$ with $W$, and so on $W$ in addition to
  the metric $\omega_{FS}$ induced by the Fubini-Study metric we have
  the metric $d_Z$ inducing the same topology. For simplicity let us also write $d_k$ for the
  metric on $M$ induced by $\omega_{t_k}$, and $M_k$ for the metric
  space $(M, d_k)$. Thus we have $M_k \to
  (W, d_Z)$ in the Gromov-Hausdorff sense. The maps $F_k :
  M_k\to\mathbf{P}^N$ are compatible with the convergence in the sense
  that if $p_k \to p$ with $p_k \in M_k$ and $p \in W$, then $F_k(p_k)
  \to p$ in $\mathbf{P}^N$. 

  If $p\in W'$, then either $p\in \mathcal{R}$ or
  $p\in \mathcal{D}$. We will only deal with the case $p\in
  \mathcal{D}$ since the other case is easier. We can
  write $p = \lim p_k$ for $p_k\in M_k$, such that for a sufficiently
  small $r > 0$ the balls $B_{d_k}(p_k, r)$, scaled to unit
  size are very close in the Gromov-Hausdorff sense to the unit ball
  in a cone $\mathbf{C}^{n-1}\times \mathbf{C}_\gamma$, for large $k$.
  As discussed in \cite{Sz13_1}, based on the ideas in \cite{CDS13_2},
  this implies that we have biholomorphisms $H_k : \Omega_k \to
  B^{2n}$, where $\Omega_k\subset M_k$ contain a ball around $p_k$ of
  a fixed size, such that the metric $\widetilde{\omega}_k = r^{-2}\omega_{t_k}$ on
  $B^{2n}$ is well approximated by the standard conical metric on
  $B^{2n}$. More precisely, we have coordinates $(u, v_1,\ldots,
  v_{n-1})$ such that if we write
  \[ \eta_\gamma = \sqrt{-1}\frac{du \wedge d\bar u}{|u|^{2-2\gamma}}
  + \sqrt{-1} \sum_{i=1}^{n-1} dv_i \wedge d\bar{v}_i, \]
  then for some fixed constant $C$ (independent of $k$)
  \begin{enumerate}
    \item $\widetilde{\omega}_k = \ddbar \phi_k$ with $0\leq \phi_k\leq
      C$,     $| r^2 v_k(\phi_k)| < C$, where $v_k$ is the soliton vector
      field in this chart.
    \item $\omega_{Euc} < C\widetilde{\omega}_k$,
    \item Given any $\delta > 0$ and compact set $K$ away from
      $\{u=0\}$, we can assume (by taking $r$ above smaller and $k$
      larger if necessary), that $|\widetilde{\omega}_k -
      \eta_\gamma|_{C^{1,\alpha}} < \delta$ on $K$.
  \end{enumerate}
We will also write $\alpha_k$ for the form $\alpha$ in this chart. 

Is is shown in
\cite[Proposition 22]{CDS13_3}, the
biholomorphisms $H_k:\Omega_k \to B^{2n}$ converge to a
homeomorphism $H_\infty : \Omega_\infty \to B^{2n}$, and necessarily
$\Omega_\infty$ contains a ball $B_{d_Z}(p, \epsilon) \subset W$ for some
small $\epsilon > 0$, since all the sets $\Omega_k$ contain balls of a
uniform size around $p_k$. It follows that $\Omega_\infty$ also
contains a ball $B$ around $p$ in the topology on $W$ induced from
$\mathbf{P}^N$, and so $H_\infty$ defines a holomorphic chart on $W$
in a neighborhood of $p$. These charts can be used to define
holomorphic maps $f_{t_k} : B \to F_{t_k}(M)$, biholomorphic onto
their image, such that the $f_{t_k}$ converge to the identity map as
$k\to\infty$. In this formulation $\beta$ is given as the weak limit
of $f_{t_k}^* (F_{t_k})_*\alpha$, which in terms of our charts amounts
to saying that $\beta$ is the weak limit of the forms $\alpha_k$. In
the same vein $\omega_T$ is the weak limit of $\omega_{t_k} =
r^2\widetilde{\omega}_k$. 

  The metrics $\widetilde{\omega}_k$ satisfy the equations 
  \[ \label{eq:d1}
e^{r^{2} v_k(\phi_k)} (\ddbar \phi_k)^n = e^{-r^2t_k\phi_k- (1-t_k)\psi_k}
  \omega_{Euc}^n, \]
  for suitable local potentials $\psi_k$ of the forms $\alpha_k$
  restricted to this chart. Note that $r^{2} v_k(\phi_k)$ is a Hamiltonian
  for $v_k$ with respect to $\omega_{t_k}$. The
  bound $\omega_{Euc} < C\ddbar\phi_k$ together with $|\phi_k|, |r^2v_k(\phi_k)| < C$
  and \eqref{eq:d1}  implies an upper bound for $\psi_k$ (note that
  $t_k$ is bounded away from $1$). In addition
  since we control $\widetilde{\omega}_k$ on compact sets away from
  $\{u=0\}$, on any such set we have a lower bound for $\psi_k$ as
  well. It follows that up to choosing a further subsequence, we have
  $\psi_k \to \psi_\infty$ in $L^1_{loc}$, for some plurisubharmonic
  $\psi_\infty$, and then necessarily $\beta = \ddbar \psi_\infty$. In
  addition we can assume that $r^2t_k\phi_k$ converge uniformly to
  $T\phi_\infty$ for a continuous $\phi_\infty$ such that $\omega_T =
  \ddbar\phi_\infty$. 
It follows that
\[ -r^2v_k(\phi_k) - \log \frac{(\ddbar \phi_k)^n}{\omega_{Euc}^n} =
r^2t_k\phi_k + (1-t_k)\psi_k \]
are plurisubharmonic functions converging in $L^1_{loc}$. The bound on
$v(\phi_k)$ implies then that the limit $\omega_T^n$ gives a singular
metric on $K_{B^{2n}}$ with locally integrable potential, and
therefore by \eqref{eq:aaa3}, we have that $e^{\theta_v}\omega_T^n$
also defines such a singular metric $e^{-\tau}$. The convergence above
then shows that \eqref{eq:aa4} holds, which is what we wanted to
show. 
\end{proof}

One important conclusion that we need to draw from this is that
according to Proposition~\ref{prop:reductive} the Lie algebra
$\mathfrak{g}_{W, \beta, v}$ is reductive. In addition the twisted
Futaki invariant vanishes, $\mathrm{Fut}_{(1-T)\beta, v}(W, w)=0$, for
any $w\in \mathfrak{g}_{W, \beta, v}$. 

Let us now identify $M$ with its image $F_0(M)\subset
\mathbf{P}^N$, and write $\alpha = (F_0)_*\alpha$. 
From the above discussion, 
for each $k$, we have $F_{t_k}(M) = \rho_k(M)$ for some $\rho_k\in
GL^G$, and $\rho_k(M) \to W$, $\rho_k(\alpha) \to \beta$.  
As before, we can write
\[ \alpha = \int_{\mathbf{P}^{N^*}} [M\cap H]\,d\mu(H), \]
since $\alpha$ is a scaling of the restriction of the Fubini-Study
metric. Note that
\[ \rho_k(\alpha) = \int_{\mathbf{P}^{N^*}} [\rho_k(M)\cap \rho_k(H)]\,d\mu(H), \]
and the following lemma implies that
we can choose a subsequence of the $\rho_k$, such that the limit 
\[ \rho_\infty(H) = \lim_{k\to\infty} \rho_k(H) \]
exists for all $H\in \mathbf{P}^{N^*}$. Note that we write
$\rho_\infty(H)$ just as a notation, rather than suggesting that an
automorphism $\rho_\infty$ of $\mathbf{P}^N$ exists. 

\begin{lem} Up to choosing a subsequence, we can assume that
  $\rho_k(H)$ converges for all $H\in \mathbf{P}^{N^*}$. 
\end{lem}
\begin{proof}
  Write $\mathbf{P}^{N^*} = \mathbf{P}(V)$ for an $N+1$-dimensional vector
  space $V$. 
  Thinking of the $\rho_k$ as matrices, let us scale each of them
  in such a way that all entries are in $\{z\,:\, |z|\leq 1\}$, and at
  least one entry equals 1. We can choose a subsequence such that as
  matrices, we have
  \[ \lim_k \rho_k = \rho, \]
  where $\rho$ is not necessarily invertible. Let $W_1 =
  \mathrm{Ker}\, \rho$. For any $x\in
  \mathbf{P}(V)\setminus\mathbf{P}(W_1)$ we can then take the limit
  \[ \lim_k \rho_k(x) = \rho(x). \]
  
  Now let us restrict the $\rho_k$ to $W$, thinking of them as
  linear maps $\rho_k: W_1\to V$. Once again, taking matrix
  representatives, we can normalize each to have entries in the unit
  disk, with at least one entry equal to 1. Just as above, up to
  choosing a further subsequence, we will have a limiting, nonzero linear map
  $\rho:W_1\to V$ with kernel $W_2\subset W_1$. For $x\in
  \mathbf{P}(W_1)\setminus \mathbf{P}(W_2)$ the limit will exist as
  above.

  Repeating this process a finite number of times we will have a
  subsequence $\rho_k$ such that $\rho_k(x)$ converges for
  all $x\in \mathbf{P}(V)$. 
\end{proof}

It follows that we have
\[ \label{eq:betaint}\beta = \int_{\mathbf{P}^{N^*}} [W\cap \rho_\infty(H)]\,d\mu(H), \]
where as before it is important to note that $W$ is irreducible and 
not contained in a hyperplane. 

In the spirit of Definition~\ref{defn:aut}, for any current $\tau$ on
$\mathbf{P}^N$, 
let us denote by
$\mathfrak{g}_{W, \tau} \subset \mathfrak{sl}(N+1,\mathbf{C})$ the space of
those holomorphic vector fields $v$, which are tangent to $W$ and satisfy
$\iota_v\tau = 0$. If $\tau = [S]$, the current of integration along a
subvariety $S$, we will write $\mathfrak{g}_S =
\mathfrak{g}_{[S]}$. Note that in this case $\mathfrak{g}_S$ is simply
the Lie algebra of the stabilizer of $S$ in $SL(N+1, \mathbf{C})$. 

\begin{lem} \label{lem:stabequal}
    We can find $H_1,\ldots, H_d$ for some $d$ such that
  \[\mathfrak{g}_{W, \beta} = \mathfrak{g}_W \cap 
\bigcap_{i=1}^d \mathfrak{g}_{[W\cap \rho_\infty(H_i)]}.\]
\end{lem}
\begin{proof}
  Suppose that $v$ is a holomorphic vector field, which does not
  vanish along $W$, and let $\xi = \iota_{\bar v}
  \omega_{FS}^n$. This is an $(n,n-1)$-form such
  that $\iota_v\xi$ is a non-negative $(n-1,n-1)$-form. 
  If $A\subset T_p\mathbf{P}^N$ is a complex $(n-1)$-dimensional
  subspace, then $\iota_v\xi$ vanishes on $A$ only if $v\in A$.
  
  If $\iota_v\beta= 0$, then we have
  \[ \int_{H\in\mathbf{P}^*} \int_{W\cap \rho_\infty(H)} \iota_v\xi\,d\mu =
  0, \]
  and so for almost every $H$ we must have
  \[ \int_{W\cap \rho_\infty(H)} \iota_v\xi = 0. \]

  In particular, for almost every $H$ we must have $v\in A$ for all
  tangent planes $A = T_p(W\cap \rho_\infty(H))$ at all smooth points $p\in
  W\cap \rho_\infty(H)$. It follows that $\iota_v[W\cap \rho_\infty(H)]=0$, i.e.
  \[ \mathfrak{g}_{\beta} \subset \mathfrak{g}_{[W\cap \rho_\infty(H)]}. \]

  If we choose one such $H$, say $H_1$, it may happen that
  $\mathfrak{g}_{[W\cap \rho_\infty(H_1)]}$ is too large, i.e. there is a $w\in
  \mathfrak{g}_{[W\cap \rho_\infty(H_1)]}$ such that $\iota_w\beta\not=0$. But we
  have
  \[ \iota_w\beta = \int_{H\in \mathbf{P}^*}
  \iota_w[W\cap \rho_\infty(H)]\,d\mu, \]
  so we must have a positive measure set of $H$ for which
  $\iota_w[W\cap \rho_\infty(H)]\ne 0$. We can thus choose an $H_2$, so that we
  still have
  \[ \mathfrak{g}_{\beta} \subset \mathfrak{g}_{[W\cap \rho_\infty(H_1)]}, \]
  but $\mathfrak{g}_{[W\cap \rho_\infty(H_1)]}\cap \mathfrak{g}_{[W\cap \rho_\infty(H_2)]}$ is strictly
  smaller than $\mathfrak{g}_{[W\cap \rho_\infty(H_1)]}$. Repeating this a finite number
  of times, we obtain the required result. 
\end{proof}

It follows from this result that we can choose $H'_1,\ldots, H'_l$ for some $l$ such
that the Lie algebra of the stabilizer of the $(l+1)$-tuple $(W, W\cap \rho_\infty(H'_1),\ldots,
W\cap \rho_\infty(H'_l))$ in $GL^G$, for the action on a product of
Hilbert schemes, is equal to the $G$-invariant part of
 $\mathfrak{g}_{W, \beta}$,
and so according to Proposition~\ref{prop:reductive} it is
reductive. Using a result similar to Luna's slice theorem~\cite{Lun73} 
as in \cite[Proposition 1]{Don10} (as in \cite{CDS13_3} as well), we can therefore find a
$\mathbf{C}^*$-subgroup $\lambda \subset GL^G$ and an element $g\in
GL^G$ such that 
\[ \label{eq:Wlim} (W, W\cap \rho_\infty(H'_1),\ldots, W\cap \rho_\infty(H'_l)) = \lim_{t\to
  0} \lambda(t)g\cdot (M, M\cap H'_1,\ldots, M \cap H'_l). \]

In addition for a subset of $E \subset \mathbf{P}^{N^*}$ of measure zero,
if $H_1, \ldots, H_K \not\in E$, then the stabilizer of 
\[ (W,  W\cap\rho_\infty(H'_1),\ldots, W\cap \rho_\infty(H'_l), W\cap \rho_\infty(H_1),
\ldots  W\cap \rho_\infty(H_{K})) \]
 will still be the same as that of $(W,\beta)$, and so we can still find a corresponding
$\mathbf{C}^*$-subgroup $\lambda$ and $g\in GL^G$ which will satisfy \eqref{eq:Wlim}
as well as
\[ \label{eq:limlambda}
\begin{aligned}
W &= \lim_{t\to 0} \lambda(t)g\cdot M \\
W\cap \rho_\infty(H_i) &= \lim_{t\to 0} \lambda(t)g\cdot (M\cap H_i),
\text{ for }i=1,\ldots, K. 
\end{aligned} \] 
Note that
all of these $\lambda$ must fix $W$, but the $\lambda$ may vary as we change
the collection $(H_1,\ldots, H_K)$. 

Each of the $\mathbf{C}^*$-actions $\lambda$ is generated by a vector
field $w$ commuting with $v$, with Hamiltonian function $\theta_w$. We
will assume that $\theta_w$ is normalized so that
\[ \int_W \theta_w\,\omega_{FS}^n = 0. \]
Let us write $\Vert w\Vert =
\sup_W |\theta_w|$, although note that any two norms on the finite
dimensional space of such $w$ are equivalent. 

Because of \eqref{eq:betaint}, 
for any $\epsilon > 0$ we can choose $K$ large, and $H_1,\ldots, H_K
\not\in E$, such that no $N+1$ of the $H_i$ lie on a hyperplane in
$\mathbf{P}^{N^*}$, and for all vector fields $w$ as above we
have
\[ 
\int_W \theta_w\, \beta \wedge \omega_{FS}^{n-1} \leq \epsilon \Vert w\Vert + \frac{1}{K}
\sum_{j=1}^K \int_{W\cap \rho_\infty(H_j)} \theta_w\,\omega_{FS}^{n-1}.
\]
Applying this to the $w$ corresponding to the $\mathbf{C}^*$-action
$\lambda$ that we obtain for $(H_1,\ldots, H_K)$, we have
\[ \int_W \theta_w\, \beta \wedge \omega_{FS}^{n-1} \leq \epsilon \Vert
w\Vert + \frac{1}{K} \sum_{j=1}^K \lim_{t\to 0} \int_{\lambda(t)\cdot
  (M\cap H_j)} \theta_w\,\omega_{FS}^{n-1}. \]

Using Lemma~\ref{lem:inttheta}, and the fact that no $N+1$ of
the $H_i$ are in a hyperplane, we obtain, using also the normalization
of $\theta_w$, that
\[ \begin{aligned} 
\int_W \theta_w\, \beta \wedge \omega_{FS}^{n-1} &\leq \left( \epsilon +
  \frac{NC}{K} \right)\Vert w\Vert - \frac{K-N}{Kn} \int_{W} 
\max_W\theta_w\,\omega_{FS}^n,
\end{aligned} \]
for some fixed constant $C$. Choosing $K$ sufficiently large
(depending on $\epsilon$), we obtain a $\mathbf{C}^*$-action generated
by a vector field $w$, with Hamiltonian function $\theta_w$ as above, such
that
\[ \label{eq:b1}\int_W \theta_w\,\beta\wedge \omega_{FS}^{n-1} \leq 2\epsilon \Vert w\Vert
- \frac{1}{n}\int_W \max_W \theta_w\,\omega_{FS}^n. \]
Moreover this $\mathbf{C}^*$-action satisfies $W = \lim_{t\to 0}
\lambda(t)g\cdot M$, but not necessarily $\beta = \lim_{t\to 0}
\lambda(t)g\cdot \alpha$. Nevertheless the vector field $v$ satisfies
$\iota_v \beta = 0$ by construction.

Since $(W, (1-T)\beta)$ admits a twisted K\"ahler-Ricci soliton, we know that
\[\mathrm{Fut}_{(1-T)\beta, v}(W,w) = 0,\] 
and so 
\[ \label{eq:b2} 
\mathrm{Fut}_v(M, w) - \frac{1-T}{V}\left[ \int_W
  \theta_w e^{\theta_v}\,\omega_{FS}^n 
 + \int_W \theta_w n\beta \wedge
  \omega_{FS}^{n-1} \right] = 0.
\]
 At the same time
we are assuming that for some $s > T$,
the triple $(M,(1-s)\psi, v)$ is K-semistable, which, using
Proposition~\ref{prop:thetaform}, implies that we have 
\[\label{eq:b3} \mathrm{Fut}_v(M, w) - \frac{1-s}{V}\left[ \int_W
  \theta_w e^{\theta_v} \,\omega_{FS}^n - V\max_W\theta_w\right] \geq
0. \]

Together \eqref{eq:b2} and \eqref{eq:b3} imply
\[ \frac{s-T}{V} \int_W \theta_w e^{\theta_v} \,\omega_{FS}^n +
(1-s)\max_W \theta_w + \frac{1-T}{V}\int_W \theta_w n\beta\wedge
\omega_{FS}^{n-1} \geq 0. \]
Using also \eqref{eq:b1} we then get
\[ 0 \leq \frac{1-T}{V} 2n\epsilon \Vert w\Vert +
\frac{s-T}{V}\int_W(\theta_w - \max_W
\theta_w)e^{\theta_v}\,\omega_{FS}^n. \]
Since $s > T$ and $T < 1$, this is a contradiction if $\epsilon$ is
sufficiently small, unless $\Vert w\Vert = 0$. For this, note that
there is a uniform constant $c > 0$ such that
\[ \int_W (\max_W \theta_w - \theta_w)e^{\theta_v}\,\omega_{FS}^n \geq
c\Vert w\Vert \]
for all possible $w$ that we have, since these form a finite
dimensional space. 

It follows that we must have $\Vert w\Vert =0$, which means that
$\theta_w$ is constant on $W$. This implies that the corresponding
$\mathbf{C}^*$-action $\lambda$ is trivial, and so in fact by
\eqref{eq:limlambda} we have
\[ (W, W\cap \rho_\infty(H_1),\ldots, W\cap \rho_\infty(H_K)) = g\cdot (M,
M\cap H_1,\ldots, M\cap H_K) \]
for some $g\in SL^G$. If follows that 
\[ \label{eq:limrhok} \lim_{k\to\infty}\rho_k(H_i) = \rho_{\infty}(H_i) = g(H_i). \]
We can assume that $H_1,\ldots, H_{N+1}$ are in general position in
$\mathbf{P}^{N^*}$, and then each $\rho_k$ is determined by the
hyperplanes $\rho_k(H_i)$ for $i =1, \ldots, N+1$. In particular
\eqref{eq:limrhok} then implies that $\rho_k\to g$ in $SL^G$, which
 in turn implies that the sequence $\rho_k\in SL^G$
is bounded. If we  write
\[ \frac{1}{m} (F_k)_*\omega_{FS} = \omega_0 + \ddbar \phi_k \]
for the pullbacks of the Fubini-Study metrics to $M$ under our
embeddings $F_k$, we then have a uniform bound $|\phi_k| < C$. The
partial $C^0$-estimate implies that then we also have
\[ \omega_{t_k} = \omega_0 + \ddbar \phi_k' \]
with $|\phi'_k| < C'$ for a uniform constant, for the metrics
$\omega_{t_k}$ along the continuity path. It is then standard using
the estimates of Yau~\cite{Yau78} that we have uniform $C^{l,\alpha}$
bounds for $\omega_{t_k}$, and so we can obtain a solution of
Equation~\eqref{eq:KRScont} for $t=T$ (see also Zhu~\cite{Zhu00} for
the $C^2$-estimate in the soliton case). 

\subsection{The case $T=1$.} Suppose now that $T=1$, i.e. we can solve
Equation~\eqref{eq:KRScont} for all $t < 1$. This case is much more
similar to the work of Chen-Donaldson-Sun~\cite{CDS13_3}, since the
``current part'' of the equation disappears as $t\to 0$. The case of
K\"ahler-Ricci solitons was also studied by Jiang-Wang-Zhu~\cite{JWZ14}. We
briefly describe the argument for the sake of completeness. Just as in
the case $T < 1$, we have embeddings $F_t : M\to\mathbf{P}^N$ using
suitable orthonormal bases for $H^0(K_M^{-m})$ with respect to the metric
$\omega_t$, for some large $m$. The partial $C^0$-estimate is still
valid, in the K\"ahler-Einstein case by \cite{Sz13_1} based on the
method in \cite{CDS13_3}, and in the
soliton case due to Jiang-Wang-Zhu~\cite{JWZ14}. It follows that as
before, up to increasing $m$ and choosing a sequence $t_k \to 1$ we
have the algebraic convergence $F_{t_k}(M) \to W \in \mathbf{P}^N$ to
a normal $\mathbf{Q}$-Fano variety, homeomorphic to the
Gromov-Hausdorff limit $(Z,d_Z)$ of the sequence $(M,
\omega_{t_k})$. As before, we identify $(M, \alpha) = (F_0(M),
(F_0)_*\alpha)$ and so $(F_{t_k}(M), (F_{t_k})_*\alpha) = \rho_k\cdot
(M, \alpha)$ for $\rho_k \in SL^G$. The vector field $v$ on each
$F_{t_k}(M)$ is induced by a fixed vector field $v$ on $\mathbf{P}^N$,
which is also tangent to the limit $W$. We can also choose a further
subsequence of $t_k$ if necessary to have a weak limit
$(F_{t_k})_*\omega_{t_k} \to \omega_1$. We have the following, see
\cite[Corollary 1.4]{JWZ14}. A proof can also be given in the spirit
of the proof of Proposition~\ref{prop:limiteq}. 

\begin{prop} The pair $(W, v)$ admits a K\"ahler-Ricci soliton, and in
  fact this soliton is given by the current $\omega_1$. 
\end{prop}

It follows from \cite[Corollary 3.6]{BW14} 
that the stabilizer of $W$ in $SL^G$ is reductive, and so
we can find a $\mathbf{C}^*$-subgroup $\lambda \in SL^G$ generated by
a vector field $w$ commuting with $v$, and an elements $g\in
SL^G$ such that
\[ W = \lim_{t\to 0} \lambda(t)g\cdot M. \]
This is a special degeneration for $M$, whose central fiber is
$W$. Since $W$ admits a K\"ahler-Ricci soliton, the corresponding
Futaki invariant $\mathrm{Fut}_v(W,w) = 0$. By assumption $(M,v)$ is
K-stable, and so $W$ must be biholomorphic to $M$. This means that
$\omega_1$ is a K\"ahler-Ricci soliton on $M$, which is what we wanted
to obtain.

\section{Some applications}\label{sec:examples}
In this section, we look at some
applications of Theorem \ref{thm:main} to existence of
K\"ahler-Einstein metrics on Fano manifolds with large symmetry
groups. 

\bigskip
\noindent {\bf Toric manifolds} \\
A compact Kahler manifold $M$ of complex
dimension $n$ is {\em toric} if the compact torus $T^n$ acts by
isometries on $M$ and the extension of the action to the complex torus
$(\mathbf{C}^*)^n$ acts holomorphically with a free, open, dense
orbit.  We can then recover the following theorem of Wang-Zhu
\cite{WZ04} as a consequence of Theorem \ref{thm:main}.  
\begin{thm}
There exists a K\"ahler-Ricci soliton, which is unique up to
holomorphic automorphisms, on every toric Fano manifold. As a
consequence, there exists a K\"ahler-Einstein metric on a toric Fano
manifold if and only if the Futaki invariant vanishes.   
\end{thm}
\begin{proof}
Let $M$ be a toric manifold with $\mathrm{dim}_{\mathbf{C}}M= n$. We
wish to apply Theorem 1 with $G=T^n$ with a fixed
identification as a subgroup of $GL(N+1,\mathbf{C})$. The key
observation is that if $v$ is a toric vector field, then any
$(\mathbf{C}^*)^n$-equivariant special degeneration of $(M,v)$ is
necessarily trivial. Indeed, if $\lambda:\mathbf{C}^* \rightarrow
GL(N+1,\mathbf{C})^G$ is a test configuration and if $M_0 =
\lim_{t\rightarrow 0}\lambda(t)\cdot M$ is not in the
$GL(N+1,\mathbf{C})$-orbit of $M$, then the stabilizer of $M_0$ must
contain a $(\mathbf{C}^*)^{n+1}$.  On the other hand, since $M_0$ is
irreducible and not contained in any hyperplane, the action of this
stabilizer on $M_0$ must also be effective. This is a contradiction
since any torus acting on an $n$-dimensional normal variety cannot
have a dimension greater than $n$. The upshot is that $M_0$ must be
bi-holomorphic to $M$ and the test configuration is induced by a toric
vector field $w$ on $M$. To verify K-stability of $(M,v)$, it then
suffices to check that the modified Futaki invariant vanishes:
$\mathrm{Fut}_v(M, w)=0$, for all toric vector fields $w$ on $M$.

Next, recall that any toric manifold $M$ with an ample line bundle
corresponds to a unique (up to translations) polytope
$P\subset\mathbf{R}^n$ defined by a finite collection of affine linear
inequalities $l_j({\bf x})\geq 0$. This polytope is in fact the image
of the free $(\mathbf{C}^*)^n$ orbit in $M$ under the moment
map. Since $M$ is Fano, one can normalize the polytope so that $l_j(0)
= 1$ for all $j$. Any toric vector field can be written as $w =
\sum_{j=1}^{n}c_j z_j\frac{\partial}{\partial z_j}$ for some ${\bf c}
\in \mathbb{R}^n$ where $(z_1,\cdots,z_n)$ are the usual complex
coordinates on $(\mathbf{C}^*)^n$. In terms of the polytope data, for
a vector field $v = \sum_{j=1}^{n}a_j z_j\frac{\partial}{\partial
  z_j}$, equation \eqref{eq:modFut} then reduces
to $$\mathrm{Fut}_v(M,w) = {\bf c}\cdot\frac{\int_{P}{\bf x}~e^{{\bf
      a}\cdot{\bf x}}\,d{\bf x}}{V},$$ where $V = Vol(P)$ is the
volume of $M$. But then, as in Tian-Zhu~\cite{TZ00}, 
 by minimizing the functional $F({\bf a}) =
\int_{P}e^{{\bf a}\cdot{\bf x}}\,d{\bf x}$, one can find a vector
${\bf a}$ such that the integral on the right vanishes, and hence
$Fut_{v}(M,w)$ vanishes identically for the corresponding toric vector
field $v$.
\end{proof} 

If $M$ does not admit a K\"ahler-Einstein metric and $\alpha \in
c_1(M)$ is a K\"ahler form, then $$R(M) = \sup\{t~|~ \exists \omega\in
c_1(M) \text{ such that } Ric(\omega)=t\omega+(1-t)\alpha\},$$
provides a natural obstruction. It follows from the work of the second
author \cite{GSz09} that $R(M)$ is in fact independent of the choice of
$\alpha$. We can then recover the following result of Li \cite{Li11},
expressing $R(M)$ in terms of the corresponding polytope. 
\begin{thm}Let $M$ be toric, Fano, and $P$ be the canonical polytope
  as above with barycenter $P_C$. Let $Q$ be the the point of
  intersection of the ray $-sP_C$, $s\geq 0$ with $\partial P$. If $O$
  denotes the origin,  
$$R(M) = \frac{|QO|}{|QP_C|}$$  
\end{thm}
\begin{proof}
By the above discussion and Proposition \ref{prop:main} it is enough
to find the maximum $t$ such that $\mathrm{Fut}_{(1-t)\psi}(M,w)\geq
0$ for all toric holomorphic vector fields $w$ where $\alpha = \ddbar
\psi$. We once again write $w = \sum_{j=1}^{n}c_j
z_j\frac{\partial}{\partial z_j}$ for some ${\bf c} \in
\mathbb{R}^n$. Then the twisted modified Futaki invariant (equation
\eqref{eq:modFut}) takes the form  $$\mathrm{Fut}_{(1-t)\psi}(M,w) =
t{\bf c}\cdot P_c + (1-t)\max_{{\bf x}\in P}{\bf c}\cdot {\bf x}.$$ 
Now let the face of the polytope containing $Q$ be given by the
vanishing of the affine linear functional $l({\bf x}) := {\bf u\cdot
  x} + 1$. Note that since $l(0) = 1$, it follows from elementary
arguments that $|QO|/|QP_C| = 1/l(P_C)$. We also remark that $l(P_C)
\geq 1$.\\ 

\medskip

\noindent{\bf Claim:} For any ${\bf c} \in \mathbb{R}^n$, $$\frac{{\bf
    c}\cdot P_C}{\displaystyle\max_{{\bf x}\in P}{\bf c}\cdot {\bf x}}
\geq 1- l(P_C).$$

\noindent Assuming this, for $t\leq 1/l(P_C)$ and any holomorphic
toric vector field $w$, it is easily seen that $Fut_{(1-t)w}(M,w)\geq
0$, and hence $R(M) \geq 1/l(P_C)$. On the other hand, if $w$ is a
special holomorphic vector field corresponding to $-{\bf u} \in
\mathbb{R}^n$, then $\max_{{\bf x}\in P}(-{\bf u})\cdot {\bf x} = 1$,
and hence $$Fut_{(1-t)\psi}(M,w) = 1 - t\cdot l(P_C).$$ This is
negative when $t>1/l(P_C)$, which implies that $R(M) = 1/l(P_C)$,
completing the proof of the theorem. To prove the claim, we first
normalize ${\bf c}$ so that $\max_{{\bf x}\in P}{\bf c}\cdot {\bf x} =
1$. If we now let $\tilde l({\bf x}) = -{\bf c\cdot x} + 1$, then
$\tilde l({\bf x}) \geq 0$ for all ${\bf x} \in P$. Moreover, since
${\bf c}\cdot P_C = 1-\tilde l(P_C)$ it is enough to show that $l(P_C)
\geq \tilde l(P_C)$. Once again consider the ray $-sP_C$ with $s\geq
0$. If this does not intersect the hyperplane $\{\tilde l=0\}$, then
clearly ${\bf c}\cdot P_C \geq 0$, and hence $\tilde l(P_C) \leq 1
\leq l(P_C)$. On the other hand, suppose the ray does intersect the
hyperplane, at say a point $Q'$. Since the polytope $P$ lies entirely
on one side of the hyperplane, we have $|QP_C|<|Q'P_C|$. In fact,
since $\tilde l(0) = l(0) = 1$, $$\tilde l(P_C) =
\frac{|Q'P_C|}{|Q'O|} = \frac{|QQ'| + |QP_C|}{|QQ'| + |QO|}\leq
\frac{|QP_C|}{|QO|} = l(P_C),$$ and the claim is proved. 
\end{proof}

\bigskip
\noindent {\bf $\mathbf{T}$-varieties}\\ 
Relaxing the toric condition, we consider Fano manifolds $M$ with an
effective action of the torus $T^m$ for some $m < n = \dim M$. The
simplest case is that of a complexity-one action, where $m =
n-1$. K\"ahler-Einstein metrics on such manifolds, in particular Fano
3-folds with 2-torus actions, was studied by S\"uss~\cite{Suss13,
  Suss14}. In particular in \cite[Theorem 1.1]{Suss14} a list of 9 such
manifolds is given with vanishing Futaki invariant, 
for 5 of which it was not known whether they admit
a K\"ahler-Einstein metric or not. Using Theorem~\ref{thm:main} one only
needs to check $T$-equivariant special degenerations, and such
degenerations can be classified using combinatorial
data. \cite[Section 5]{Suss14} lists all such degenerations to
canonical toric Fano varieties, while the more general degenerations
to log-terminal toric Fanos are classified by
Ilten-S\"uss~\cite{IS15}. The conclusion is that all 9 Fano threefolds
with vanishing Futaki invariant in \cite[Theorem 1.1]{Suss14}
 admit a K\"ahler-Einstein metric.

\bigskip
\noindent {\bf Other manifolds with large symmetry group.}\\
We expect that Theorem~\ref{thm:main} can be used to show the
existence of K\"ahler-Einstein metrics on many other classes of Fano
manifolds with large symmetry group. One interesting class is that of 
reductive varieties, studied by
Alexeev-Brion~\cite{AB04, AB04_1}.
 Let $G$ be a connected compact group, $T \subset G$ a
maximal torus, and $W$ the corresponding Weyl group. Denote by
$\Lambda$ the character group of $T$, which is a lattice in the real
vector space $\Lambda_{\mathbf{R}}$. To every $W$-invariant maximal dimensional
convex lattice polytope $P\subset \Lambda_{\mathbf{R}}$ one can
associate a variety $V_P$, which is a $G^c\times G^c$-equivariant
compactification of $G^c$, the action being left and right multiplication. As
shown in~\cite{AB04_1} (see also Alexeev-Katzarkov~\cite{AK05}), 
the equivariant degenerations of $V_P$ correspond to convex, rational, $W$-invariant, piecewise
linear functions $f$ on $P$, in analogy to the toric case studied in
Alexeev~\cite{Ale02}, Donaldson~\cite{Don02}. If we have an
equivariant special degeneration, then in particular the central fiber
is irreducible, and this will only happen when $f$ is linear on $P
\cap \Lambda_{\mathbf{R}}^+$, where $\Lambda_{\mathbf{R}}^+\subset
\Lambda_{\mathbf{R}}$ is a positive Weyl chamber corresponding to a Borel subgroup of $G^c$,
containing $T^c$. It follows that there are only a finite number of
degenerations that need to be checked in order to apply
Theorem~\ref{thm:main}. 

In the case when $P\cap \Lambda_{\mathbf{R}}^+$ is a maximal set on
which $f$ is linear, then the central fiber of the corresponding
special degeneration is a horospherical variety. These are the
homogeneous toric bundles studied by Podesta-Spiro~\cite{PS10_1}, who
showed that all such Fano manifolds admit a K\"ahler-Ricci
soliton. This also follows from the above discussion together with our
main result, since the polytope $P$ can not be subdivided further, and
so a horospherical variety has no non-trivial
equivariant special degenerations, just as the toric manifolds
discussed above.

\section{The partial $C^0$-estimate for solitons} \label{sec:partialC0}
In this section we briefly outline the changes that have to be made to
the arguments in \cite{Sz13_1}, using also techniques in
Zhang~\cite{Z10}, Tian-Zhang~\cite{TZ12} and
Phong-Song-Sturm~\cite{PSS12}, 
to prove the partial $C^0$-estimate for
the family of metrics $\omega_t\in c_1(M)$ solving
\[ \label{eq:bb1} \mathrm{Ric}(\omega_t) - L_v \omega_t = t\omega_t + (1-t)\alpha, \]
where $t\in [0, T)$ with $T < 1$. The case when $T=1$ has been
established by Jiang-Wang-Zhu~\cite{JWZ14}. Here $v$ is a holomorphic
vector field, such that $\mathrm{Im}\,v$ generates a compact torus of
isometries of the metric $\alpha$. In particular $\omega_t$ will also
be invariant under this torus. To simplify notation, we will drop the
subscript $t$, and so in what follows, $\omega$ denotes a solution of
\eqref{eq:bb1} for some $t\in [0,T)$. 

Recall that we have the Hamiltonian function $\theta_v$ of $v$, with
respect to the metric $\omega$, defined by
\[ \iota_v \omega = \sqrt{-1}\dbar \theta_v, \]
with the normalization
\[ \int_M e^{\theta_v}\,\omega^n = \int_M \omega^n. \]
From Zhu~\cite{Zhu00}, and Wang-Zhu~\cite[Lemma 6.1]{WZ13_1} we know
that we have estimates
\[ \label{eq:thetaest}|\theta_v| + |\nabla \theta_v|_{\omega} + |\Delta_{\omega}
\theta_v| < C. \]

The Equation~\eqref{eq:bb1} implies that
\[ \label{eq:BElower}  \mathrm{Ric}(\omega) - L_v \omega \geq
0. \]
In addition as soon as $t$ is bounded away from 0, the volume
comparison and Myers type theorem in Wei-Wylie~\cite{WW09} implies
that the diameter of $(M,\omega)$ is bounded, and we have the
non-collapsing property 
\[ \label{eq:noncollapsed} \mathrm{Vol}(B(p,1),\omega) \geq c > 0. \] 
There are two basic approaches to studying metrics satisfying
this lower bound for the Bakry-\'Emery Ricci curvature, generalizing
the theory of Cheeger-Colding~\cite{CC97} in the case when $v=0$. 
One approach is to study the conformally related
metrics $\widetilde{g}_{j\bar k } = e^{-\frac{1}{n-1}\theta_v} g_{j\bar
  k}$, where $g_{j\bar k}$ is the metric with K\"ahler form $\omega$. 
 This approach, similar to that used in Zhang~\cite{Z10} and
Tian-Zhang\cite{TZ12} (who used the Ricci potential instead of $\theta_v$), 
effectively reduces the problem to studying
non-collapsed metrics with a lower Ricci curvature bound so that the
theory of Cheeger-Colding can be applied. Indeed, in real coordinates
the Ricci tensor of $\widetilde{g}$ satisfies
\[ \label{eq:rictilde}\widetilde{R}_{ij} = R_{ij} + \nabla_i\nabla_j \theta_v +
\frac{1}{2(n-1)} \nabla_i \theta_v \nabla_j\theta_v -
\frac{1}{2(n-1)}\big[ |\nabla \theta_v|^2_g - \Delta_g\theta_v\big]
g_{ij}, \]
and so \eqref{eq:BElower}, \eqref{eq:thetaest} together with the fact that $v$ is
holomorphic, and so $\nabla_i\nabla_j \theta_v$ is of type $(1,1)$,
imply that $\widetilde{g}$ has a Ricci lower bound. In addition it is
clear that $\widetilde{g}$ is uniformly equivalent to $g$. 
The other approach is to
build up the Cheeger-Colding theory using the bound \eqref{eq:BElower}
on the Bakry-Emery Ricci curvature. This approach is executed by
Wang-Zhu~\cite{WZ13_1}. We summarize the main conclusions from these works
that we need. 

If we have a sequence $(M, \omega_i)$, satisfying \eqref{eq:thetaest},
\eqref{eq:BElower} and \eqref{eq:noncollapsed}, then 
up to choosing a subsequence, the Riemannian manifolds $(M, g_i)$
converge in the Gromov-Hausdorff sense to a length space $(Z, d)$. At
each point $p\in Z$ there exists a tangent cone $C(Y)$ which is a
metric cone. We can stratify the space $Z$ as
\[ S_n \subset S_{n-1} \subset \ldots \subset S_1 = S \subset Z, \]
where $S_k$ consists of those points, where no tangent cone is of the
form $\mathbf{C}^{n-k+1}\times C(Y)$. 

The regular part of $Z$ is defined to be $\mathcal{R} = Z\setminus
S$, and at $p\in \mathcal{R}$ every tangent cone is $\mathbf{C}^n$.
We also write $\mathcal{D} = S \setminus S_2$. The
following is analogous to Anderson's regularity result~\cite{An90},
showing that we have good control of the metrics on the regular set if
we also have an upper bound of the Bakry-\'Emery Ricci curvature. 

\begin{prop}\label{prop:Anderson} Suppose that $B(p,1)$ is a unit ball in K\"ahler
  manifold $(M,\omega)$, together with a holomorphic vector field $v$ with
  Hamiltonian $\theta$, satisfying bounds of the form
\begin{enumerate}
  \item $\sup_M |\theta| + |\nabla\theta| + |\Delta\theta| < K$\\
  \item $0\leq \mathrm{Ric}(\omega) -
    L_v\omega \leq K\omega$. 
\end{enumerate}
There are constants $\delta, \kappa > 0$ depending on $K$ such that
if $d_{GH}(B(p,1), B^{2n}) < \delta$, then for each $q\in
B(p,\frac{1}{2})$, the ball $B(q,\kappa)$ is the domain of a
holomorphic coordinate system in which the components of $\omega$
satisfy
\[ \begin{gathered} \frac{1}{2}\delta_{jk} < \omega_{j\bar k} < 2\delta_{jk}, \\
  \Vert \omega_{j\bar k}\Vert_{L^{2,p}} < 2, \text{ for all }p. 
\end{gathered}\] 
\end{prop}

\begin{proof}
We use the conformal scaling  $\widetilde{g} =
e^{-\frac{1}{n-1}\theta} g$, so that by \eqref{eq:rictilde} $\widetilde{g}$ satisfies
  two-sided Ricci curvature bounds. 
  Suppose that $d_{GH}((B(p,1),g),
  B^{2n}) < \delta$. The bound on $\nabla\theta$
  implies that if $q\in B(p,\frac{1}{2})$ and $r$ is sufficiently small, then
  \[d_{GH}( (B(q,r), \widetilde{g}), r\lambda B^{2n}) < 2\delta,\]
  for a suitable
  scaling factor $\lambda$ (depending on the value $\theta(q)$).

  If $\delta$ is sufficiently small, then
  Colding's volume convergence result~\cite{Col97} combined with Anderson's gap
  theorem implies that there is a harmonic coordinate system on the
  ball $B(q, r\theta\lambda, \widetilde{g})$ in which the metric
  $\widetilde{g}$ is controlled in $L^{2,p}$ for any $p$. The metrics
  $\widetilde{g}$ and $g$ are $C^1$-equivalent, so we also control the
  components of $g$ in $C^1$. The Laplacian bound on $\theta$ then
  implies that we have $L^{2,p}$ estimates on $\theta$ so in
  fact $g$ and $\widetilde{g}$ are equivalent in $L^{2,p}$. In
  particular in our harmonic coordinates (harmonic for
  $\widetilde{g}$) we control the coefficients of $g$ in
  $L^{2,p}$. Using that the complex structure is covariant constant,
  this allows us to find holomorphic coordinates on a possibly smaller
  ball, in which the coefficients of $g$ are controlled in $L^{2,p}$. 
\end{proof}

Following Chen-Donaldson-Sun, define
\[ I(\Omega) = \inf_{B(x,r)\subset\Omega} VR(x,r), \]
where $\Omega$ is any domain in a K\"ahler manifold, and $VR(x,r)$ is
the ratio of volumes of the ball $B(x,r)$ in $\Omega$ and the
Euclidean ball $rB^{2n}$. If the Ricci
curvature is non-negative, 
the Bishop-Gromov comparison theorem and Colding's volume
convergence implies that if $B$ is a unit ball in $\Omega$, then
$1-I(B)$ controls $d_{GH}(B, B^{2n})$,
and conversely $d_{GH}(B, B^{2n})$ controls $1-I(B)$. In our setting,
with the bound \eqref{eq:BElower}, a similar statement will only hold
once the metrics are scaled up by a sufficient amount. We have the
following. 

\begin{prop}\label{prop:Colding}
Suppose that $B$ is a unit ball in a K\"ahler manifold $(M,\omega)$
satisfying 
\[ Ric(\omega) - L_v \omega \geq 0, \]
as well as 
\[ \sup_B |\nabla \theta| + |\Delta\theta| \leq \delta, \]
where $\theta$ is a Hamiltonian of $X$. Then
\[ d_{GH}(B, B^{2n}) = \Psi(\delta, 1-I(B)), \]
and for any $\lambda < 1$, 
\[ 1- I(\lambda B) = \Psi(\delta, d_{GH}(B, B^{2n}), 1-\lambda), \]
where $\Psi(\epsilon_1, \ldots, \epsilon_k)$ denotes a function
converging to zero as $\epsilon_i\to 0$. 
We have suppressed the dependence of $\Psi$ on the
dimension $n$. 
\end{prop}
\begin{proof}
  We can assume that $\theta(0) =
  0$. 
  Use the conformal metric $\widetilde{g} = e^{-\frac{1}{n-1}\theta}
  g$. Then under our assumptions we have $Ric(\widetilde{g}) > -C'\delta
  \widetilde{g}$ and the metric $\widetilde{g}$ is very close in $C^0$
  to the metric $g$. We can
  then apply the volume convergence under lower Ricci curvature bounds
  to the metric $\widetilde{g}$. 
\end{proof}

We now return to our original setup, of a metric $\omega$ on $M$ satisfying
\[ \label{eq:omegai}
Ric(\omega) - L_v\omega = t\omega + (1-t)\alpha, 
\]
for some  $t\in [0,T)$, and $T < 1$. The vector field $v$ and
background metric $\alpha$ is fixed. As before we can assume
that the metrics are non-collapsed, and in addition the Hamiltonian
$\theta_v$ of $v$ satisfies
\[ \sup_M (|\nabla \theta_v|^2 + |\Delta \theta_v|) \leq K, \]
for some fixed constant $K$. The square is inserted for scaling
reasons. 
Note that for any point $p\in M$ we can choose the $\theta_v$ so that
$\theta_v(p)=0$. 
We will exploit the fact that $\alpha$ is a fixed metric. In
particular we can assume that $K$ is chosen such that on any ball of radius at
most $K^{-1}$ with respect to $\alpha$ we can find holomorphic
coordinates in which the coefficients of $\alpha$ are controlled in
$C^2$. 

To understand the tangent cones of the Gromov-Hausdorff limit of a
sequence of metrics satisfying these conditions, we need
to study very small balls in $(M, \omega)$, scaled up to unit
size.  Let $(B,\eta)$ be a small ball in $(M,\omega)$
 scaled to unit size, so that $\eta = \Lambda\omega$ for
some large $\Lambda$. Let $w = \Lambda^{-1}v$. Then
$\eta$ satisfies 
\[ Ric(\eta) - L_{w}\eta = \lambda \eta + (1-t)\alpha, \]
for some $\lambda\in (0,1]$ and $t\in (0, T)$.  In addition we can
choose the Hamiltonian $\theta_w$ for $w$ relative to
$\eta$
such that $\theta_w(0)=0$, and
\[ \sup_M (|\nabla \theta_w|_\eta^2 + |\Delta_\eta \theta_w|) \leq
\Lambda^{-1}K.\]

The following is the generalization of Proposition 8 in \cite{Sz13_1},
showing that on the regular set the Gromov-Hausdorff limit behaves as
if we had a two-sided Ricci curvature bound. Note
that as in Proposition~\ref{prop:Colding} we need an extra assumption
ensuring that we have scaled our metrics up by a sufficient amount. 
\begin{prop}
  There is a $\delta > 0$ depending on $K$ above, such that
  if $1-I(B) < \delta$, and the scaling factor $\Lambda > \delta^{-1}$ then 
 \[ \alpha < 4\eta\, \text{ in } \frac{1}{2}B. \]
\end{prop}
\begin{proof}
  The method of proof is the same as in \cite{Sz13_1}. Suppose that 
\[ \sup_B d_x^2 |\alpha(x)|_\eta = M, \]
where $d_x$ is the distance of $x$ to the boundary of $B$ with respect to $\eta$, and suppose that
the supremum is achieved at $q\in B$. If $M > 1$ then we can consider
the ball
\[ B\left( q, \frac{1}{2} d_q M^{-1/2}\right), \]
scaled to unit size $\widetilde{B}$, with scaled metric
$\widetilde{\eta} = 4Md_q^{-2}\eta$. Note that $\widetilde{\eta}$ satisfies the
same estimates as $\eta$, but in addition
$|\alpha|_{\widetilde{\eta}} \leq 1$ on $\widetilde{B}$. If
$\delta$ is sufficiently small, then 
we can apply Propositions~\ref{prop:Anderson} and \ref{prop:Colding}
to find holomorphic coordinates $z_i$ on a small ball $\tau\widetilde{B}$, in which the
components of $\widetilde{\eta}$ are controlled in
$C^{1,\alpha}$. 

The metric $\widetilde{\eta}$ satisfies 
\[ \begin{aligned}
\mathrm{Ric}(\widetilde{\eta}) &= L_w\eta + \lambda\eta + (1-t)\alpha
\\
&\geq  (4Md_q^{-2})^{-1} L_{w}
\widetilde{\eta} 
+ (1-t)\alpha,
\end{aligned}\]
and for any $\epsilon > 0$ we can choose the scaling factor $\Lambda$
large enough, so that the Hamiltonian of $w$ satisfies
$|\nabla\theta_w|^2_\eta < \epsilon$, which implies
$|w|^2_{\widetilde{\eta}} < 4Md_q^{-2} \epsilon$. Since $w$ is a
holomorphic vector field, we obtain that in the coordinates $z_i$, on
the half ball $\frac{\tau}{2}B$, the
components of $w$, along with their derivatives are bounded by
$(4Md_q^{-2}\epsilon)^{1/2}$. It follows that on this ball we have
\[ |L_w \eta|_{\widetilde{\eta}} < C \epsilon^{1/2}
(4Md_q^{-2})^{-1/2}, \]
for some fixed constant $C$.  
In particular if $\delta$ is chosen
sufficiently small, then we  will have $L_w\eta <
\epsilon\widetilde{\eta}$ and so
\[ \mathrm{Ric}(\widetilde{\eta}) \geq
-\epsilon\widetilde{\eta} + (1-t)\alpha.\]
Using this, the rest of the proof is essentially identical to that in \cite{Sz13_1}. 
\end{proof}

Together with Proposition~\ref{prop:Anderson} it follows from this
that in the Gromov-Hausdorff limit of a sequence of metrics
$\omega$ satisfying \eqref{eq:omegai}, with $t < T < 1$, the
regular set is open and smooth, and the convergence of the metrics 
is $C^{1,\alpha}$ on the regular set. In addition the same holds for iterated
tangent cones.

What remains is to study tangent cones of the form
$\mathbf{C}_\gamma\times\mathbf{C}^{n-1}$, i.e. the points in the set
$\mathcal{D}$ in the Gromov-Hausdorff limit. The arguments in
\cite[Proposition 11, 12, 13]{Sz13_1} can be followed closely with a couple
  of remarks. First of all the results of Chen-Donaldson-Sun~\cite{CDS13_2} on good
  tangent cones can be applied. The main difference here is that a
  variant of the $L^2$-estimates in \cite[Proposition 2.1]{DS12} needs
  to be used, following \cite[Proposition 4.1]{PSS12}, with the
  Hamiltonian $\theta_v$ replacing the Ricci potential $u$. 
This implies that if a scaled up ball $(B, \eta)$
  as above is sufficiently close to the unit ball in
  $\mathbf{C}_\gamma \times \mathbf{C}^{n-1}$, then on a smaller ball
  we have holomorphic coordinates, in which the metric $\eta$
  satisfies the conditions $(1), (2), (3)$ in the proof of
  Proposition~\ref{prop:limiteq}. 

 An additional important fact used several times is that by
  Cheeger-Colding-Tian~\cite{CCT02}, no tangent cone of the form
  $\mathbf{C}_\gamma\times \mathbf{C}^{n-1}$ can form in the
  Gromov-Hausdorff limit of a sequence of K\"ahler metrics with
  bounded Ricci curvature. The analogous result with the bound on Ricci
  curvature replaced by a bound on $\mathrm{Ric}(\omega) - L_v \omega$ was shown
  by Tian-Zhang~\cite{TZ12}, and it also follows from the more recent
  work of
  Cheeger-Naber~\cite{CN14} in the general Riemannian case. With these
  observations the proof of the partial $C^0$-estimate for solutions
  of \eqref{eq:bb1} follows the argument in \cite{Sz13_1} closely.

\section{Reductivity of the automorphism group and vanishing of the Futaki invariant}\label{sec:reductive}
In this section we briefly outline the proofs of
Proposition~\ref{prop:uniqueness} and
Proposition~\ref{prop:reductive} following \cite{Ber13},\cite{CDS13_3} and
\cite{BW14}. As before, let $W$ be the normal $\mathbb{Q}$-Fano
variety obtained as the Gromov-Hausdorff limit along the continuity
method, and $v\in H^0(W,TW)$ such that $Im(v)$ generates the action of
a torus $T$ on $W$. We let $\mathcal{H}_v$ denote the space of
continuous $T$-invariant metrics $h_\phi=e^{-\phi}$ on $-K_W$ with
non-negative curvature. Then the twisted Ding functional is defined
as \[\mathcal{D}_{(1-t)\psi,v}(\phi) =
-tE_v(\phi)-\log{\Big(\int_{W}e^{-t\phi-(1-t)\psi}\Big)},\] where
$E_v$ is defined by its variation at $\phi$ in the direction 
$\dot\phi$ by 
\[ \frac{d}{ds} E_v(\phi) = \frac{1}{V} \int_W \dot{\phi}
e^{\theta_v}\,\omega_\phi^n, \]
as in Berman-Witt-Nystr\"om~\cite{BW14}. 
Next, we recall the definition of a geodesic in the path of K\"ahler
metrics. We let $\mathcal{R} = \{s 
\in \mathbb{C}~|~Re(s) \in [0,1]\}$. Recall that a path $\phi_s\in
\mathcal{H}_v$ is called a geodesic if $\Phi : W\times \mathcal{R}
\rightarrow \mathbb{R}$ defined by $\Phi(x,s) = \phi_{Re(s)}(x)$
satisfies$\sqrt{-1}\partial\bar\partial_{s,W}(\Phi) \geq 0$
and $$(\sqrt{-1}\partial\bar\partial_{s,W}\Phi)^{n+1} = 0,$$  where
the $\partial\bar\partial$ is taken in both $W$ and $\mathcal{R}$
directions. Then the following is proved in \cite{Ber13} 
\begin{lem}\label{geodesics}
For any $\phi_0,\phi_1 \in \mathcal{H}_v$, there exists a geodesic
$\phi_s\in \mathcal{H}_v$ connecting them such that $$||\phi_{s'} -
\phi_s||_{L^{\infty}(W)} < C|s'-s|$$ 
\end{lem}The key point is that the Ding functional is convex along
these geodesics. It is proved in Berman-Witt-Nystr\"om \cite[Proposition
2.17]{BW14} that the functional $E_v(\phi)$ is affine along
geodesics and continuous up to the boundary. So the convexity of the
Ding functional is a consequence of the following result of Bendtsson
\cite{Ber13}. 
\begin{prop}\label{convexity}Let $\phi_s$ be a geodesic as above. Then
  the functional  
\[\mathcal{F}(s) = -\log\Big(\int_{W}e^{-t\phi_s-(1-t)\psi}\Big)\]
is convex. Moreover, if $\mathcal{F}(s)$ is affine, then there exists a holomorphic vector fields $w_s$ on $W$ with $i_{w_s}\ddbar\psi=0$,  and such that the flow $F_s$ satisfies
\begin{align*}
F_s^*(\ddbar\phi_s) &= \ddbar\phi_0 \\
\end{align*}
\end{prop}
This was proved on compact K\"ahler manifolds by Berndtsson
\cite{Ber13} and extended to normal varieties by Chen-Donaldson-Sun
\cite{CDS13_3} when $\ddbar\psi$ is the current of integration along a
divisor (see also~\cite{BBEGZ11}). Though the above statement does not seem to follow directly
from either of the works, the arguments can be easily adapted, and we
briefly provide an outline of the proof.  
\begin{proof} For ease of notation, we let $\tau_s = t\phi_s + (1-t)\psi$.
Let $p:W'\rightarrow W$ be a log-resolution. and $\omega'$ be a fixed
K\"ahler metric on $W'$. Since $W$ has only log terminal
singularities, one has the following adjunction formula  
\begin{equation}\label{adjunction}
-K_{W'} = -p^*K_W - E + \Delta,
\end{equation}
where $E$ and $\Delta$ are effective divisors, and $\Delta =
\sum{a_jE_j}$ with $a_j\in (0,1)$. Suppose first that $e^{-\tau_s}$ is
a smooth family of metrics on $-K_W$, inducing a smooth family of
pull-back metrics on $-p^*K_W$ with curvature $\omega_{\tau_s'} =
\ddbar\tau_s'$. We write $L=K_{W'}^{-1}\otimes
E$. Then from \eqref{adjunction} it is clear that $$\tau'_s =
p^*\tau_s + \sum a_j\log{|s_j|^2},$$ where $s_j$ is the defining
function of $E_j$, induces a family of singular metrics $e^{-\tau'_s}$
on $L$. Moreover, if $u$ is a holomorphic $L$-valued $(n,0)$ form with zero
divisor $E$ (which is unique up to multiplication by a constant) it
can be easily checked that up to scaling $u$ by a constant, 
\begin{align*}
\mathcal{F}(s) = -\log\int_{W'}u\wedge \bar u \,e^{-\tau'_s}.
\end{align*}

Let us pretend for the moment that the metrics $e^{-\tau_s'}$ are
smooth. Consider the equation
\[\label{L2 problem}\nabla_s \nu_s =
P_s\Big(\frac{d\tau'_s}{ds}u\Big),\]
where $\nabla_s = \partial
- \partial \tau'_s \wedge\cdot$ is the Chern connection of
$e^{-\tau'_s}$ and $P_s$ is the projection onto the orthogonal
complement of $L$-valued holomorphic $(n,0)$ forms. As argued in
\cite{Ber13}, it can be shown that there always exists a smooth
solution $\nu_s$ to \eqref{L2 problem} satisfying
$\bar\partial\nu_s\wedge \omega' = 0$. Next, the Hessian of
$\mathcal{F}$ is given by (\cite[Theorem 3.1]{Ber13},\cite[Lemma
14]{CDS13_3}) \[||u||^2_{\tau'_s}\ddbar\mathcal{F}(s) =
\int_{W'}\omega_s'\wedge \tilde u \wedge \bar{\tilde u}~e^{-\tau_s'}
~+~ ||\dbar \nu_s||^2_{\tau_s'} \sqrt{-1}ds\wedge d\bar s,\] where
$\tilde u = u - ds\wedge \nu_s$ and $\omega_s' =
\ddbar_{s,W'}(\tau'_s)$. This is in fact a special case of the general
positivity of direct image sheaves discovered by Bendtsson
\cite{Ber09}. For smooth geodesics, the convexity follows directly
from this formula.

In our case the metrics $\tau_s'$ are not smooth,
and hence we first need to use a regularization. First, if we let
$\eta = \omega' + \sqrt{-1}ds\wedge d\bar s$, then by the
approximation theorem of Demailly~\cite{Dem92} (see also
Blocki-Kolodziej~\cite{BK07}) there exists a decreasing sequence
of smooth metrics $\rho_{s,\epsilon}\searrow p^*\tau_s$ such that
$\ddbar_{s,W'}(\rho_{s,\epsilon}) \geq -C\eta$. By averaging we can
also suppose that $\rho_{s,\epsilon}$ are independent of $Re(s)$ and
$T$-invariant. To approximate $\tau_s'$ we then let
$\tau_{s,\epsilon}' = \rho_{s,\epsilon} + \log{h_\epsilon}$
where \[\log{h_\epsilon} = \sum{a_j (\log{(|s_j|^2_{h_j} + \epsilon)}
  - \log{h_j})}\]and $h_j$ is a metric on the line bundle generated by
$E_j$. Clearly $e^{-\tau_{s,\epsilon}'}$ are metrics on $L$ with
$\tau_{s,\epsilon}'\searrow \tau'_s$ and
$\ddbar_{s,W'}(\tau_{s,\epsilon}') >-C\eta$ for some $C>0$. Moreover,
for any neighborhood $U$ of $\Delta$ there exists a constant $C_U$
such that
$$\ddbar_{s,W'}(\tau_{s,\epsilon}') >-\epsilon C_U\eta,\text{ on } W'
\setminus U. $$
We then let $\nu_{s,\epsilon}$ be the solutions to \eqref{L2 problem}
corresponding to $\tau_{s,\epsilon}$. The key point now is the
following lemma of Berndtsson which guarantees uniform estimates for
these solutions independent of $s$ and $\epsilon$.  
\begin{lem}\cite[Lemmas 6.3,6.5]{Ber13},\cite[Lemmas 17,19]{CDS13_3}
\begin{itemize}
\item There exists a constant $C$ (independent of $s,\epsilon$) such
  that $$\Vert\nu_{s,\epsilon}\Vert_{L^2({\tau'_{s,\epsilon}})}\leq
  C\left\Vert\frac{d\tau'_{s,\epsilon}}{ds}u\right\Vert_{L^2({\tau'_{s,\epsilon}})}$$ 
\item For every $\delta$-neighborhood $U_\delta$ of $\Delta$, there
  exists a constant $c_\delta$ such that $c_\delta\rightarrow 0$ as
  $\delta \rightarrow 0$
  and $$\int_{U_\delta}|\nu_{s,\epsilon}|^2_{\tau'_{s,\epsilon}} \leq
  c_\delta\Big(\int_{W'}|\nu_{s,\epsilon}|^2_{\tau'_{s,\epsilon}} +
  |\bar\partial \nu_{s,\epsilon}|^2_{\tau'_{s,\epsilon}}\Big)$$ 
\end{itemize}
\end{lem}
Note that the norms of $\nu_{s,\epsilon}$ also
involve a K\"ahler metric on $W'$ which we take to be the fixed metric
$\omega'$. We also remark that this was proved by Berndtsson for
metrics $e^{-\xi}$ where $\xi$ is only upper bounded, and hence is
applicable in our situation since $\tau_{s,\epsilon}'$ are easily seen
to be upper bounded. Once we have this uniform $L^2$ estimate, the
rest of the argument in \cite{CDS13_3} can be followed almost
verbatim. That is, if we write for $\mathcal{F}_{\epsilon}(s)$ for the
functional corresponding to $\tau_{s,\epsilon}'$, then
$\mathcal{F}_\epsilon \searrow \mathcal{F}$. Moreover, using the
Hessian formula above one can show that for any $r\in(0,1)$ on
$[r,1-r]$ we have $$\frac{d^2\mathcal{F}_\epsilon}{ds^2} >
-c_{\epsilon} \rightarrow 0.$$ This shows that $\mathcal{F}$ is indeed
convex.

Suppose now that $\mathcal{F}$ is affine linear.
Observe that since $\tau_{s,\epsilon}'$ decrease to
$\tau_s'$ and $\tau_{s,\epsilon}'$ are uniformly Lipschitz in $s$,
$||\nu_{s,\epsilon}||_{L^2({\tau'_{s,\epsilon}})}$ are uniformly
bounded. Hence $\nu_{s,\epsilon}$ converges weakly in $L^2(\tau_{s}')$
to an $L$-valued $(n-1,0)$ form $\nu_s$  with $\bar\partial\nu_s =
0$. Integrating by parts, it can be shown that $\nu_s$ solves
\eqref{L2 problem} weakly on $W'\setminus\{\psi=-\infty\}$ or
equivalently, $\nabla_s\nu_s-u~d\tau_s'/ds$
is holomorphic on $\{\psi\neq\infty\}$, and it is in $L^2$. But since pluripolar
sets are removable for $L^2$ holomorphic forms,
$\nabla_s\nu_s-u~d\tau_s'/ds$ is also holomorphic globally. Using the formula
$\bar\partial\nabla_s \nu_s + \nabla_s\bar\partial \nu_s =
\omega_{\tau_s'}\wedge \nu_s$ it follows
that \[\label{limit1}\omega_{\tau_s'}\wedge\nu_s =
\sqrt{-1}\bar\partial\Big(\frac{d\tau_s'}{ds}\Big)\wedge u.\] 
A family of holomorphic vector fields $w_s'$ can now be defined on
$W'\setminus E$ by $$\iota_{w_s'}u = \nu_s,$$ so that away from $E$ we
have $\iota_{w_s'}\omega_{\tau'_s} = -\sqrt{-1}~\dbar\dot{\tau'_s}$. Then
$w_s = p_*w_s'$ is a holomorphic vector field on $W_{0}$ which by
normality of $W$ extends to a global time-dependent holomorphic vector
field on $W$. Next, note that $p^{-1}$ is a biholomorphism when
restricted to $W_{0}$, and $\omega_{\tau_s} =
(p^{-1})^*\omega_{\tau_s'}$. It then follows that on $W_{o}$,
$\iota_{w_s}\omega_{\tau_s} = -\sqrt{-1}~\dbar\dot{\tau_s}$ and
hence, \[\label{Lie1}\mathcal{L}_{w_s}\omega_{\tau_s} =
-\frac{\partial}{\partial s}\omega_{\tau_s},\] as currents. Moreover,
it can be shown that $\partial w_s/\partial\bar s= 0$, and hence $w_s$
generates a holomorphic flow $F_s$ (see \cite[Lemma 5.2]{BBEGZ11}).
Also, note that $w_s'$ has uniform
$L^2$ bound (independent of $s$) away from $E$, and hence the flow
$F_s$ extends continuously to $s=0,1$ such that $F_0$ is the
identity. From \eqref{Lie1} it follows that on
$W_{0}$,
$$\frac{\partial}{\partial s}F_s^*\omega_{\tau_s} =
F_s^*\Big(\frac{\partial}{\partial s}\omega_{\tau_s} +
\mathcal{L}_{w_s}\omega_{\tau_s}\Big) = 0.$$
In particular
$F_s^*\omega_{\tau_s}=  \omega_{\tau_0}$ on $W_{0}$, and hence
globally on $W$ by unique extension of closed positive $(1,1)$
currents over sets of Hausdorff co-dimensions greater than two. Now,
if we define a holomorphic vector field $\mathcal{W}_s
= \partial/\partial s - w_s$ on $W\times \mathcal{R}$, following the
same line of argument as in \cite[Lemma 4.3]{Ber13} we can show
that $$\iota_{\mathcal{W}_s}\ddbar_{s,W}(\tau_s) = 0.$$ Again following
\cite{Ber13} $$0=\iota_{\overline{\mathcal{W}_s}}\iota_{\mathcal{W}_s}\ddbar_{s,W}{(\tau_s)}
=
t~\iota_{\overline{\mathcal{W}_s}}\iota_{\mathcal{W}_s}\ddbar_{s,W}{(\phi_s)}
+ (1-t)\iota_{\overline{w_s}}
\iota_{w_s}\ddbar{\psi}.$$ Since both the $(1,1)$
currents on the right are non-negative, each has to be zero. Again,
since $\ddbar\psi\geq 0$, by Cauchy's inequality for any $(1,0)$
vector field $\xi$, $\iota_{\overline{\xi}}\iota_{w_s}\ddbar\psi = 0$, and
hence $\iota_{w_s}\ddbar\psi = 0$. In particular,
$\mathcal{L}_{w_s}\ddbar\psi = 0$, and hence $F_s^*\ddbar\phi_s =
\phi_0$, which completes the proof of the proposition.  
\end{proof}

\subsection*{Proof of Proposition~\ref{prop:uniqueness}.} Let $e^{-\phi_0}$ and
$e^{-\phi_1}$ be two soliton metrics on $(W,(1-t)\psi,v)$ and
$\phi_s\in\mathcal{H}_v$ be a bounded geodesic connecting $\phi_0$ and
$\phi_1$. Since solitons are the stationary points of
$\mathcal{D}_{(1-t)\psi,v}$, the one sided derivatives at $s=0$ and
$s=1$ (which exist by convexity of the Ding functional) are zero. As a
consequence $\mathcal{D}_{(1-t)\psi,v}(\phi_s)$, and hence
$\mathcal{F}(s)$, is affine, and by Proposition \ref{convexity} there
exists a family of holomorphic vector fields $w_s$ with flow $F_s$
such that $F_s^*\omega_{\phi_s} = \omega_{\phi_0}$. Next, note that
$\phi_j$ for $j=0,1$ satisfies 
 \[\label{soliton1}Ric(\omega_{\phi_j}) = t\omega_{\phi_j} +
 (1-t)\ddbar\psi + \mathcal{L}_v\omega_{\phi_j}\] on $W_{0}$. So on
 the one hand, since $\phi_s$ are stationary points of
 $\mathcal{D}_{(1-t)\psi,v}$, $\omega_{\phi_s}$ also satisfies
 \eqref{soliton1}, while on the other hand $\omega_{\phi_s}$ satisfies
 \eqref{soliton1} with $v$ replaced by $(F_s)_*v$. Hence if we set
 $\xi_s = (F_s)_*v - v$, then
 $\mathcal{L}_{\xi_s}\omega_{\phi_s}=0$. This implies that if $h_s$ is
 the hamiltonian of $\xi_s$ with respect to $\omega_{\phi_s}$, then
 $\ddbar h_s = 0$ and consequently $v=(F_s)_*v$. To show the
 time-independence of the vector fields, arguing as in the proof of
 \cite[Proposition 4,5]{Ber13}, we can show that $$\iota_{(F^{-1}_s)_*w_s
   - w_0} \omega_{\phi_0} = 0.$$ Since $\phi_0$ is bounded, and hence
 in particular $e^{-\phi_0}$ is integrable, by
 Berndtsson~\cite[Proposition 8.2]{Ber13} the above equation forces
 $(F^{-1}_s)_*w_s = w_0$. This shows that the vector fields are
 independent of time, and in fact $F_s$ is just the flow generated by
 $w_0$. Finally since $\iota_{w_0}\omega_{\phi_0} =
 -\sqrt{-1}~\dbar\dot{\phi_0}$ and $\phi_0$ is real valued, $Im(w_0)$
 is also a Killing field for $\omega_{\phi_0}$. This completes the
 proof of the proposition with $w=w_0$. 

 \subsection*{Proof of Proposition~\ref{prop:reductive}.} As shown
 in \cite{CDS13_3}
reductivity follows from uniqueness, and we reproduce their
arguments. Suppose $\omega$ is the twisted K\"ahler-Ricci soliton on
the triple $(W,(1-t)\psi,v)$, and let $H$ be the connected group with
Lie algebra $\mathfrak{g}_{W, \psi, v}$ naturally identified as a subgroup of $SL(N+1,\mathbf{C})$. Let
$K\subset H$ be the subgroup of isometries of $\omega$ with the corresponding Lie sub-algebra of $\mathfrak{g}_{W, \psi, v}$ given by  $$\mathfrak{k}_{W, \psi, v} = \{w\in H^{0}(W,T^{1,0}W)~:~
\mathcal{L}_{Re(w)} \omega= 0,~\iota_w\omega_\psi = 0,~[w,v]=0\},$$
which can naturally be identified as a sub-algebra of $\mathfrak{su}(N+1,\mathbf{C})$. Moreover, since the trace form on $\mathfrak{su}(N+1,\mathbf{C})$ given by $B(x,y)=\mathrm{tr}(xy)$  is negative definite. it's restriction to $\mathfrak{k}_{W, \psi, v}$ is a non-degenerate bilinear form, and hence $\mathfrak{k}_{W, \psi, v}$ is a reductive Lie algebra. Next, if $K^c\subset SL(N+1,\mathbf{C})$  is the connected complexification of $K$, then clearly $K^c \subset H$. Conversely, for any $h\in
H$, it can be checked that $h^*\omega$ is also a twisted
K\"ahler-Ricci soliton for the triple $(W,(1-t)\psi,v)$, and hence by
Proposition 4 there exists an element $F \in K^c$ such that $h^*\omega
= F^*\omega$. But then $h\circ F^{-1}\in K$, and hence $H = K^c$. As a consequence $\mathfrak{g}_{W, \psi, v} = \mathfrak{k}_{W, \psi, v}\otimes_{\mathbf{R}}\mathbf{C}$, and is reductive. The same proof suitably modified shows that the centralizer $\mathfrak{g}_{W, \psi, v}^G$ is also reductive.

\subsection*{Proof of Proposition~\ref{prop:Futvanish}.}
Suppose that $e^{-\phi}$ is a smooth metric on $K_{W}^{-1}$, and
$f_t\in \mathrm{Aut}(W)$ is a one-parameter group of
biholomorphisms, generated by $w\in \mathfrak{g}_{W, \psi,v}$.
 In particular since $f_t^*\omega_\psi = \omega_\psi$,
we must have $f_t^*(e^{-\psi}) = c_t e^{-\psi}$ for some constants
$c_t$. Similarly to \cite[Lemma 12]{CDS13_3}, we consider the quantity
\[ \begin{aligned}
  I(e^{-\phi}) &= \frac{1}{V}\int_W \log \frac{ \left(\int_W
    e^{-\phi}\right)^{-1} e^{-\phi}}{ \left(\int_W e^{-t\phi -
      (1-t)\psi}\right)^{-1} e^{-t\phi - (1-t)\psi}} \omega_\phi^n \\
&= \log \frac{\int_W e^{-t\phi - (1-t)\psi}}{\int_W e^{-\phi}}  -
\frac{1-t}{V}\int_W (\phi - \psi)\omega_\phi^n,
\end{aligned}\]
where we note that $\phi-\psi$ is a globally defined integrable
function. 
We have $I(f_t^*(e^{-\phi})) = I(e^{-\phi})$, and differentiating this
at $t=0$ we obtain (using $\dot\phi = \theta_w$), that
\[ \begin{aligned}
  &\frac{\int_W \theta_w e^{-\phi}}{\int_W e^{-\phi}} - \frac{ \int_W
    t\theta_w e^{-t\phi - (1-t)\psi}}{\int_W e^{-t\phi - (1-t)\psi}} -
  \frac{1-t}{V} \int_W \theta_w \omega_\phi^n \\
  &\quad - n\frac{1-t}{V}\int_W (\phi-\psi) \ddbar\theta_w \wedge
  \omega_\phi^{n-1} = 0.
\end{aligned}\]
Integrating by parts in the last integral, and using the definition
\eqref{eq:Futdef} of the twisted Futaki invariant, we obtain
\[ \label{eq:Futformula2}\mathrm{Fut}_{(1-t)\psi, v}(W,w) = \frac{t}{V}\int_W \theta_w
e^{\theta_v} \omega_\phi^n - t\frac{\int_W \theta_w e^{-t\phi -
    (1-t)\psi}}{ \int_W e^{-t\phi - (1-t)\psi}}. \]
 Note that this formula is not well defined if $e^{-(1-t)\psi}$ is not
integrable, but we only need it in that case, since by assumption $(W,
(1-t)\psi, v)$ admits a twisted K\"ahler-Ricci soliton. 

By the convexity of $\mathcal{D}_{(1-t)\psi, v}$, 
 twisted K\"ahler-Ricci solitons minimize the twisted Ding
functional, we know that $\mathcal{D}_{(1-t)\psi, v}$ is bounded
below. At the same time \eqref{eq:Futformula2} implies that
\[ \frac{d}{dt} \mathcal{D}_{(1-t)\psi, v}(f_t^*\phi) =
-\mathrm{Fut}_{(1-t)\psi, v}(W, w), \]
and as a result the twisted Futaki invariant must vanish.

\subsection*{Acknowledgements}
We would like to thank Valery Alexeev, Robert
Berman, Duong Phong,  Jian Song, Jacob Sturm, and Hendrik S\"uss for helpful
discussions. The second named author is supported by 
National Science Foundation grants DMS-1306298 and DMS-1350696.

\end{document}